\documentclass[12pt]{scrartcl}

\usepackage[utf8]{luainputenc}
\usepackage[USenglish]{babel}
\usepackage{csquotes}

\usepackage[a4paper,top=27mm,bottom=20mm,inner=25mm,outer=20mm]{geometry}

\usepackage[%
  backend=bibtex,bibencoding=ascii,
  style=numeric-comp,
  giveninits=true, uniquename=init, %abbreviate first names
  natbib=true,
  url=true,
  doi=true,
  isbn=false,
  backref=false,
  maxnames=99,
  ]{biblatex}
\usepackage{amsmath}
\allowdisplaybreaks
\numberwithin{equation}{section}
\usepackage{amssymb}
\usepackage{commath}
\usepackage{mathtools}
\usepackage{bbm}
\usepackage{nicefrac}
\usepackage{subdepth}

\usepackage{siunitx}

\usepackage{amsthm}
\usepackage{thmtools}
\usepackage{etoolbox}
\makeatletter
\patchcmd{\thmt@setheadstyle}
 {\bgroup\thmt@space}
 {\thmt@space}
 {}{}
\patchcmd{\thmt@setheadstyle}
 {\egroup\fi}
 {\fi}
 {}{}
\makeatother
\declaretheoremstyle[
  bodyfont=\normalfont\itshape,
  headformat=\NAME\ \NUMBER\NOTE,
]{myplain}
\declaretheoremstyle[
  headformat=\NAME\ \NUMBER\NOTE,
]{mydefinition}
\newcommand{\envqed}{{\lower-0.3ex\hbox{$\triangleleft$}}}
\declaretheorem[style=myplain,numberwithin=section]{theorem}

\declaretheorem[style=mydefinition,numberlike=theorem,qed=\envqed]{definition}

\usepackage[plainpages=false,pdfpagelabels,hidelinks,unicode]{hyperref}

\usepackage{color}
\usepackage{graphicx}
\usepackage[small]{caption}
\usepackage{subcaption}

\begingroup\expandafter\expandafter\expandafter\endgroup
\expandafter\ifx\csname pdfsuppresswarningpagegroup\endcsname\relax
\else
\fi

\usepackage{booktabs}
\usepackage{rotating}
\usepackage{multirow}

\usepackage{enumitem}

\usepackage{ifluatex}
\ifluatex
  \usepackage[no-math]{fontspec}
\else
  \usepackage[T1]{fontenc}
\fi
\usepackage{newpxtext,newpxmath}

\usepackage{xparse}

\let\epsilon\varepsilon
\let\phi\varphi
\let\rho\varrho

\usepackage{xspace}

\newcommand{\eg}[0]{{e.g.\@}\xspace}
\newcommand{\ie}[0]{{i.e.\@}\xspace}

\renewcommand{\O}{\mathcal{O}}

\newcommand{\R}{\mathbb{R}}
\renewcommand{\H}{\mathcal{H}}
\newcommand{\inv}{\mathcal{I}}
\renewcommand{\S}{S}
\newcommand{\xint}{\int_{\xmin}^{\xmax}}

\newcommand{\I}{\operatorname{I}}

\DeclarePairedDelimiterX\newset[1]\lbrace\rbrace{\setaux #1||\endsetaux}
\def\setaux#1|#2|#3\endsetaux{\if\relax\detokenize{#2}\relax #1 \else #1 \;\delimsize\vert\; #2 \fi}
\renewcommand{\set}[1]{\newset*{#1}}

\newcommand{\dt}{\Delta t}

\renewcommand{\vec}[1]{\pmb{#1}}
\newcommand{\mat}[1]{\vec{#1}}
\NewDocumentCommand{\D}{m+g}{%
  \IfNoValueTF{#2}
    {D_{#1}}
    {D_{#1,#2}}%
}

\NewDocumentCommand{\A}{m+g}{%
  \IfNoValueTF{#2}
    {A_{#1}}
    {A_{#1,#2}}%
}
\NewDocumentCommand{\M}{g}{%
  \IfNoValueTF{#1}
    {M}
    {M_{#1}}%
}
\NewDocumentCommand{\proj}{g}{%
  \IfNoValueTF{#1}
    {P}
    {P_{#1}}%
}
\NewDocumentCommand{\eL}{g}{%
  \IfNoValueTF{#1}
    {\vec{e}_{L}}
    {\vec{e}_{L,#1}}%
}
\NewDocumentCommand{\eR}{g}{%
  \IfNoValueTF{#1}
    {\vec{e}_{R}}
    {\vec{e}_{R,#1}}%
}
\NewDocumentCommand{\dL}{g}{%
  \IfNoValueTF{#1}
    {\vec{d}_{L}}
    {\vec{d}_{L,#1}}%
}
\NewDocumentCommand{\dR}{g}{%
  \IfNoValueTF{#1}
    {\vec{d}_{R}}
    {\vec{d}_{R,#1}}%
}

\newcommand{\xmin}{x_\mathrm{min}}
\newcommand{\xmax}{x_\mathrm{max}}

\newcommand{\orcid}[1]{ORCID:~\href{https://orcid.org/#1}{#1}}
\usepackage{authblk}

\newenvironment{keywords}{\par\textbf{Key words.}}{\par}
\newenvironment{AMS}{\par\textbf{AMS subject classification.}}{\par}

\title{On the Rate of Error Growth in Time for Numerical Solutions of Nonlinear Dispersive Wave Equations}
\author[1]{Hendrik~Ranocha\thanks{\orcid{0000-0002-3456-2277}}}
\author[2]{Manuel~Quezada~de~Luna\thanks{\orcid{0000-0001-7288-0367}}}
\author[2]{David~I.~Ketcheson\thanks{\orcid{0000-0002-1212-126X}}}
\affil[1]{%
Applied Mathematics,
University of Münster,
Germany}
\affil[2]{%
King Abdullah University of Science and Technology (KAUST),
Computer Electrical and Mathematical Science and Engineering Division (CEMSE),
Thuwal, 23955-6900, Saudi Arabia}
\date{September 8, 2021}

\makeatletter
\hypersetup{pdftitle={\@title}}
\makeatother

\begin{document}

\maketitle

\begin{abstract}
We study the numerical error in solitary wave solutions of nonlinear dispersive
wave equations.  A number of existing results for discretizations of solitary
wave solutions of particular equations indicate that the error grows
quadratically in time for numerical methods
that do not conserve energy, but grows only linearly for conservative methods.
We provide numerical experiments suggesting that this result extends to a
very broad class of equations and numerical methods.

\end{abstract}


%TODO: MSC

\section{Introduction}

In this work we study the behavior of numerical approximations of solitary wave
solutions of nonlinear wave partial differential equations (PDEs).  These
solitary waves evolve in time only by translation; this translation symmetry
is, in accordance with Noether's Theorem
\cite{Noether1918}, associated with a particular conserved functional of the PDE.
We investigate in a broad setting the phenomenon, already studied for certain
equations, that numerical methods that conserve (up to rounding errors) this
functional (or the Hamiltonian of the system) exhibit much smaller errors
over long times.

This behavior has been demonstrated for numerical approximations of
certain solutions of Hamiltonian systems.  These solutions are referred
to as relative equilibria, and are those whose trajectory lies
on a manifold defined by a symmetry group.
For such solutions, numerical methods that exactly conserve certain
invariants give rise to a leading term in the global error that grows
linearly in time \cite{duran1998numerical}.  In contrast, for general (non-conservative) numerical
solutions approximating relative equilibria, the error growth is typically
quadratic (this should be contrasted with more general estimates of numerical
error growth, which are typically exponential in time).  This behavior has been shown for
finite-dimensional Hamiltonian systems
\cite{cano1997error,cano1998error,duran1998numerical,calvo2011error}.
For dispersive nonlinear wave equations, it was first observed in the
context of solitary wave solutions of the Korteweg-de Vries (KdV) equation, by Sanz-Serna \cite{sanzserna1982explicit}
and later by Bona, Dougalis, \& Karakashian \cite{bona1986fully}.
This behavior was eventually proved for discretizations of the KdV equation by
de Frutos \& Sanz-Serna \cite{frutos1997accuracy}.

Similar results were later obtained for solitary wave solutions of
the nonlinear Schrödinger equation \cite{duran2000numerical}, Peregrine's regularized
long wave (RLW) equation (also known as the Benjamin-Bona-Mahoney (BBM)
equation) \cite{araujo2001error,momoniat2014modified}, and the generalized BBM
equation \cite{duran2002numerical,alvarez2012preservation}.  An extension
to certain other dispersive wave equations, including the generalized
KdV, generalized BBM, and generalized Benjamin-Ono equations, is presented
in \cite{alvarez2012preservation}.  The results in
\cite{momoniat2014modified} are purely experimental, while
the remaining results just cited include proofs along with numerical
demonstrations.

Most of the PDE results in this direction are specific to solitary wave
solutions.  The conservative
discretizations that yield linear error growth also exhibit {\em modified solitary waves}
\cite{araujo2001error,duran2002numerical}.
These are discrete solitary waves that are close to but different from the exact
solitary wave solutions.  The ability of these modified solitary waves
to propagate discretely without changing shape or amplitude means that
conservative methods are appealing for long-time simulations of solitary waves
and their interactions \cite{duran2003conservative}.  The concept of
a modified solitary wave also provides, for some systems, an intuitive rough explanation of
the observed error growth rates.  In a non-conservative method, the solitary
wave is approximated by a wave whose amplitude grows or diminishes linearly
in time.  Since the speed of the solitary wave is (at least to first approximation)
a linear function of the amplitude, the error in the speed of the wave also grows
linearly in time.  This leads to a quadratic error in the location (phase) of the wave.
Meanwhile, it was shown in \cite{araujo2001error,duran2002numerical,alvarez2012preservation} that
a conservative scheme yields a modified solitary wave whose amplitude and speed
are constant in time (though different from those of the true solitary wave).
Hence the phase error grows only linearly.

The foregoing results serve as the motivation for the present study.
Here we take an experimental approach, allowing the consideration of
several other dispersive wave equations as well as non-dispersive
hyperbolic systems that possess similar kinds of solutions.
In this work we provide strong evidence that the phenomenon of linear
error growth in time for conservative numerical solutions
extends to a very large class of dispersive wave equations and nonlinear
invariants, including invariants that are not quadratic, and even certain
solutions of first-order hyperbolic systems that have no dispersive terms.

We briefly introduce the spatial and temporal discretization methods we will
use in Section~\ref{sec:discretizations}.
In Section~\ref{sec:dispersive-wave-eqs}, we analyze the temporal error growth
of solitary wave solutions to several nonlinear dispersive wave equations.
Our experimental results include the
Fornberg-Whitham \cite{whitham1967variational},
Camassa-Holm \cite{camassa1993integrable},
Degasperis-Procesi \cite{degasperis2002new},
Holm-Hone \cite{holm2003nonintegrability},
and BBM-BBM equations \cite{bona2002boussinesq,bona2004boussinesq}.
For each of these models, we apply recently-constructed conservative schemes
and show experimentally that they yield improved (linear) error growth when
compared to non-conservative schemes for the same model.
To demonstrate the importance of nonlinearity for the different error growth
behaviors of conservative and non-conservative methods, we investigate a
linear dispersive wave equation in Section~\ref{sec:linear}.  There, both kinds
of methods result in the same asymptotic error growth rate, although the absolute
error of the conservative method is smaller.
We then go further outside the realm of available results and consider two
first-order nonlinear hyperbolic systems that have been shown to possess
solitary-wave-like solutions.  In Section~\ref{sec:p-system} we study the
variable-coefficient $p$-system in one dimension, with periodically varying
coefficients.  Previous work on this system has noted similarities with
dispersive nonlinear wave equations; see \cite{leveque2003solitary,2012_ketchesonleveque_periodic}.
In Section~\ref{sec:swe}, we study the two-dimensional shallow water equations
with varying bathymetry.
Finally, we summarize and discuss our results in Section~\ref{sec:summary}.

The main contributions of this work are as follows:
\begin{enumerate}
    \item We illustrate via computational experiments that the favorable behavior proved
        for conservative numerical solutions of certain nonlinear dispersive
        PDEs \cite{frutos1997accuracy,duran2000numerical,araujo2001error,alvarez2012preservation}
        holds for several other such PDEs.
    \item We show through numerous examples that not only do conservative schemes
        exhibit a better asymptotic error behavior for large times, but also a much
        smaller size of the global error even for short times.
    \item We show through a computational example that the condition
        of mass conservation, assumed in the main theorems of
        \cite{frutos1997accuracy,araujo2001error}, is necessary in
        practice (see Section~\ref{sec:proj}).
    \item We demonstrate for the first time that the aforementioned favorable
        behavior can be observed also in the
        numerical solution of a system that does not possess solitary wave solutions
        with simple translation or rotation symmetries (see Section~\ref{sec:p-system}).
    \item We demonstrate for the first time that dramatically better error
        growth for conservative methods can be observed also in systems with two
        spatial dimensions (see Section~\ref{sec:swe}).
\end{enumerate}

The source code used to generate the results of this study is
available online \cite{ranocha2021rateRepro}.
The numerical methods studied in this article are implemented in Julia
\cite{bezanson2017julia} and make use of the open source library
SummationByPartsOperators.jl \cite{ranocha2021sbp}, which utilizes
FFTW \cite{frigo2005design} for Fourier collocation methods.
We use time integration methods from \cite{rackauckas2017differentialequations}.
The plots are generated using Matplotlib \cite{hunter2007matplotlib}.

\section{Relative equilibrium solutions}\label{sec:rel-eq}
One framework for understanding the rates of error growth observed
in the present work is that of relative equilibrium theory.
Here we review very briefly this theory in the context of
finite-dimensional dynamical systems (where it is most completely developed)
and Hamiltonian PDEs.

Relative equilibria are equilibria of a reduced dynamical system obtained by
reduction modulo a set of symmetry groups of the original dynamical system.
Relative equilibria evolve only along the manifold corresponding to the
symmetry groups.  In the case of finite-dimensional Hamiltonian
systems, each symmetry group corresponds to an invariant quantity.
A relative equilibrium $u_0$ of a Hamiltonian system with Hamiltonian $\H(u)$
and invariant quantities $\inv_j(u)$ satisfies \cite[Eqn. (36)]{duran1998numerical}
\begin{align} \label{releq}
    \nabla \left( \H(u_0) - \sum_i \lambda_i \inv_i(u_0)\right) & = 0,
\end{align}
where the values $\lambda_i$ are Lagrange multipliers.
For a relative equilibrium solution of a Hamiltonian system, perturbations
in directions that lie on the manifold corresponding to the symmetry group
grow linearly in time, while perturbations in other directions grow quadratically
\cite[Lemma~2.4]{duran1998numerical}.  Therefore, if a numerical method conserves
the invariants corresponding to the symmetry group, the numerical error will
grow only linearly in time \cite[Theorem~3.2]{duran1998numerical}.

The situation is analogous but more involved for Hamiltonian PDEs like the
dispersive wave equations considered in Section~\ref{sec:dispersive-wave-eqs}.
Each of these equations can be written as a Hamiltonian system
\begin{align} \label{hamiltonian-form}
    u_t = J \delta \H(u)
\end{align}
where $\H$ is the Hamiltonian, $\delta$ is the variational derivative,
and $J$ is a skew-symmetric operator.
Each system possesses an additional invariant $\inv(u)$, such that
\begin{align} \label{translation-generator}
    J\delta \inv(u) = \partial_x u.
\end{align}
This implies that the symmetry group of translations is associated
with the invariant $\inv$; the relative equilibrium condition
(analogous to \eqref{releq}) is
\begin{align} \label{vareq}
    \delta (\H(u) - c \inv(u)) & = 0.
\end{align}
Using \eqref{hamiltonian-form} and \eqref{translation-generator},
we find that multiplying \eqref{vareq} on the left by $J$ gives the
equation $\partial_t u - c\partial_x u = 0$, satisfied by translation-invariant
solitary waves with speed $c$.  Thus, for each of the equations we
consider, solitary wave solutions correspond to relative equilibria, where the
symmetry group is generated by translation.

It is possible to show also for these systems that
numerical methods whose errors lie entirely within the appropriate manifold
yield an error that grows only linearly in time; see \eg
\cite{frutos1997accuracy,duran2000numerical,alvarez2012preservation} for examples of such analysis.
Due to the condition \eqref{vareq}, it is sufficient to conserve either
$\H$ or $\inv$, in order to obtain linear growth in time of the leading term in the
global error \cite{frutos1997accuracy,alvarez2012preservation}.
Rigorous results in this vein require a detailed perturbation analysis specific
to the PDE in question.  We do not pursue such analysis here, and focus
instead on an experimental study.

We remark here that the results proving linear error growth for KdV and RLW
solitary waves \cite{frutos1997accuracy,araujo2001error} also require the assumption of mass conservation.
The necessity of this assumption has not been investigated experimentally,
perhaps since ``sensible methods exactly conserve the mass" \cite{frutos1997accuracy}.
However, the well-known orthogonal projection schemes can be used to
conserve a nonlinear invariant at the cost of not conserving mass; in
Section~\ref{sec:proj} we investigate the asymptotic error growth for such schemes.

\section{Discretization methods}
\label{sec:discretizations}

Next, we introduce the type of discretization methods
employed in this article.

\subsection{Spatial semidiscretizations}

Proving the conservation of invariants of partial differential equations
usually requires the product/chain rule and integration by parts. It is
often useful to mimic the same procedure at the discrete level by using
summation by parts (SBP) operators, which are constructed specifically to
satisfy a discrete analogue of integration by parts. A review of the relevant
theory can be found in \cite{svard2014review,fernandez2014review,chen2020review}.
Many classes of numerical methods can be formulated within the SBP framework,
including finite difference \cite{strand1994summation},
finite volume \cite{nordstrom2001finite,nordstrom2003finite},
continuous Galerkin \cite{hicken2016multidimensional,hicken2020entropy,abgrall2020analysisI},
discontinuous Galerkin \cite{gassner2013skew},
and flux reconstruction methods \cite{ranocha2016summation}.
A brief review of how to formulate these methods in the SBP framework with
application to structure-preserving numerical methods used in this article
is given in \cite{ranocha2021broad}.

Next, we will briefly introduce SBP operators in periodic domains in one space
dimension. Extensions to multiple space dimensions can be achieved via tensor
products. We consider a grid $\vec{x} = (x_1, \dots, x_m)$ where
$x_1 = \xmin \le x_2 \dots \le x_{m} = \xmax$. All nonlinear operations will
be performed pointwise, \ie we use a collocation approach.
\begin{definition}
  Given a grid $\vec{x}$, a \emph{$p$-th order accurate $i$-th
  derivative matrix} $\D{i}$ is a matrix that satisfies
  \begin{equation}
    \forall k \in \{0, \dots, p\}\colon \quad
      \D{i} \vec{x}^k = k (k-1) \dots (k-i+1) \vec{x}^{k-i},
  \end{equation}
  with the convention $\vec{x}^0 = \vec{1}$ and $0 \vec{x}^k = \vec{0}$.
  We say $\D{i}$ is consistent if $p \ge 0$.
\end{definition}

For periodic boundary conditions (under whichsolutions at $\xmin$ and $\xmax$ are identical),
integration by parts is basically a statement about the symmetry of a derivative
operator with respect to the $L^2$ scalar product. Discretely, an approximation
of this scalar product is represented by a so-called mass matrix $\M$.\footnote{The name \emph{mass matrix} is common in the finite element literature. In classical articles on finite difference SBP operators, this matrix is often called a \emph{norm matrix}.}
\begin{definition}
\label{def:D1-periodic}
  A \emph{periodic first-derivative SBP operator}
  consists of
  a grid $\vec{x}$,
  a consistent first-derivative matrix $\D1$,
  and a symmetric and positive-definite matrix $\M$
  such that
  \begin{equation}
  \label{eq:D1-periodic}
    \M \D1 + \D1^T \M = 0.
  \end{equation}
\end{definition}
We will often refer to an operator $\D{i}$ as a (periodic) SBP operator
if the other operators (such as the mass matrix $\M$) are clear from the context.
We always assume derivative operators are consistent, but we will
usually omit this term.

\begin{definition}
\label{def:D2-periodic}
  A \emph{periodic second-derivative SBP operator}
  consists of
  a grid $\vec{x}$,
  a consistent second-derivative matrix $\D2$,
  and a symmetric and positive-definite matrix $\M$
  such that
  \begin{equation}
  \label{eq:D2-periodic}
    \M \D2 = -\A2,
    \quad
    \A2 \text{ is symmetric and positive semidefinite}.
  \end{equation}
\end{definition}

\begin{definition}
\label{def:D4-periodic}
  A \emph{periodic fourth-derivative SBP operator}
  consists of
  a grid $\vec{x}$,
  a consistent fourth-derivative matrix $\D4$,
  and a symmetric and positive-definite matrix $\M$
  such that
  \begin{equation}
  \label{eq:D4-periodic}
    \M \D4 = \A4,
    \quad
    \A4 \text{ is symmetric and positive semidefinite}.
  \end{equation}
\end{definition}

In this work, we use Fourier (pseudospectral) collocation methods
\cite{kreiss1972comparison,fornberg1975fourier} because of their
efficiency and accuracy for smooth problems.  This allows us to
ensure, with reasonable grids, that the temporal discretization error
dominates the spatial error for all of our 1D numerical tests.  Nevertheless, all classes of
methods within the SBP framework can be used and analyzed interchangeably to
construct conservative semidiscretizations.

\subsection{Time integration methods}

To transfer the semidiscrete conservation results to fully-discrete schemes,
the recent relaxation approach is used
\cite{ketcheson2019relaxation,ranocha2020relaxation,ranocha2020relaxationHamiltonian,ranocha2020general,ranocha2020fully}.
Related ideas date back to \cite{sanzserna1982explicit,sanzserna1983method} and
\cite[pp. 265--266]{dekker1984stability} but have been developed widely just
recently.

Semidiscretizations reduce an initial boundary value PDE to an initial value ODE
\begin{equation}
  u'(t) = f(u(t)), \quad u(0) = u^0.
\end{equation}
Throughout this work we use superscripts to denote the time step, and here
$u^0$ is the initial condition.
To conserve a nonlinear invariant functional $J(u)$ discretely in a one-step method,
we require $J(u^n) = J(u^{n-1}) = J(u^0)$. For a Runge-Kutta method
\begin{subequations}
\label{eq:RK-step}
\begin{align}
\label{eq:RK-stages}
  y_i
  &=
  u^n + \dt \sum_{j=1}^{s} a_{ij} \, f(t_n + c_j \dt, y_j),
  \qquad i \in \set{1, \dots, s},
  \\
\label{eq:RK-final}
  u(t_n + \dt) \approx u^{n+1}
  &=
  u^n + \dt \sum_{i=1}^{s} b_{i} \, f(t_n + c_i \dt, y_i),
\end{align}
\end{subequations}
we define the update direction
\begin{align}
  d^n := \sum_{i=1}^{s} b_{i} f_i,
\end{align}
where we use the abbreviation $f_i := f(t_n + c_i \dt, y_i)$.
Since the new solution $u^n$ will not be conservative in general,
we modify the update formula \eqref{eq:RK-final} by introducing a relaxation
parameter $\gamma^n$ and use
\begin{equation}
\label{eq:u-gamma}
  u(t_n + \gamma^n\dt) \approx u^{n+1}_\gamma = u^n + \gamma^n \dt d^n.
\end{equation}
To guarantee conservation of $J$, $\gamma^n$ is computed as root of
\begin{equation}
\label{eq:gamma}
  J(u^{n+1}_\gamma) = J(u^n).
\end{equation}
This scalar nonlinear equation can be solved efficiently using algorithms such
as the one of \cite{alefeld1995algorithm}. By the general theory on relaxation
methods, there is exactly one root $\gamma^n = 1 + \O(\dt^{p-1})$ of
\eqref{eq:gamma} \cite[Theorem~2.14]{ranocha2020general}.
By construction of the relaxation parameter $\gamma^n$, the resulting solution
$u^{n+1}_\gamma$ conserves the invariant $J$. Moreover, the relaxation approach
also automatically conserves linear invariants (as long as the semi-discretization
conserves them), which is an important prerequisite of analytical results on
linear vs. quadratic error growth in time. Finally, the solution \eqref{eq:u-gamma}
has the same local order of accuracy as that given by the original Runge-Kutta
method \eqref{eq:RK-step}.

\section{Nonlinear dispersive wave equations}
\label{sec:dispersive-wave-eqs}

In this section, we consider several nonlinear dispersive wave equations.
Each possesses an additional invariant $\inv(u)$, and possesses stable
solitary wave solutions that satisfy the relative equilibrium condition
\eqref{vareq}.  Each equation also possesses one or more linear invariants,
which are preserved by the discretizations we use.  Since these invariants
do not require any special numerical treatment, we do not discuss them
explicitly.

For each equation we give a numerical method that conserves $\inv(u)$,
based on SBP operators applied to a split form of the equation.  In
each case, $D_j$ always denotes a periodic $j$th-derivative
SBP operator.  The matrices used for discretization of a given equation
are also assumed to have in common a corresponding diagonal mass matrix $M$.
We will sometimes require an assumption that the derivative matrices $D_j$
commute; this will be stated explicitly when it is required.

Conservative numerical methods for these equations have been developed and
analyzed in \cite{ranocha2021broad} using various classes of SBP operators.
We will use split forms to preserve local conservation laws, since the chain
and product rules cannot hold discretely for many high-order discretizations
\cite{ranocha2019mimetic}. These are related to entropy-conservative
methods in the sense of Tadmor \cite{tadmor1987numerical}. While certain split
forms have been known for some time \cite[eq. (6.40)]{richtmyer1967difference},
they are still state of the art and enable the construction of numerical methods
with desirable properties
\cite{gassner2016split,shima2021preventing,winters2018comparative,ranocha2018generalised}.

For each equation, after reviewing the Hamiltonian structure and split-form
conservative spatial discretization, we present a numerical test of error
growth for a solitary wave solution.
The solitary wave solutions are obtained via the Petviashvili method
\cite{petviashvili1976equation} using a Fourier collocation method with
$2^{16}$ nodes. The resulting solitary wave profile is interpolated to
a grid using fewer nodes and used as initial condition. A conservative
semidiscretization is obtained by using Fourier collocation methods
in space.  This semidiscretization is integrated in time using the
fifth-order accurate Runge-Kutta method of Tsitouras \cite{tsitouras2011runge}
with adaptive time stepping based on local error estimates.  We show results
for both this method without relaxation (non-conservative) and with relaxation
in time (conservative).
The errors plotted are discrete $L^2$ errors computed using the discrete
norm induced by the mass matrix $\M$, which is the identity matrix times the grid spacing for Fourier collocation methods.

\subsection{Fornberg-Whitham equation}
\label{sec:fw}

For convenience we define $\S = \I - \partial_x^2$.  Then
the Fornberg-Whitham equation \cite{whitham1967variational}
\begin{equation}
\label{eq:fw-dir}
\begin{aligned}
  \S \partial_t u(t,x)
  + \S \partial_x f(u(t,x))
  + \partial_x u(t,x)
  &= 0,
  %&& t \in (0, T), x \in (\xmin, \xmax),
  \\
  u(0, x) &= u^0(x),
  %&& x \in [\xmin, \xmax],
  \\
  f(u) &= \frac{u^2}{2},
\end{aligned}
\end{equation}
with periodic boundary conditions can also be written as
\begin{equation}
\label{eq:fw-inv}
  \partial_t u(t,x) + \partial_x f(u(t,x)) + \S_P^{-1} \partial_x u(t,x) = 0,
\end{equation}
where $\S_P^{-1}$ is the inverse of the elliptic
operator $\I - \partial_x^2$ with periodic boundary conditions.
%In addition to the linear invariants
%\begin{subequations}
%\label{eq:fw-invariants}
%\begin{align}
%\label{eq:fw-invariants-linear}
%  & \int_{\xmin}^{\xmax} u, &
%  & \int_{\xmin}^{\xmax} (u - \partial_x^2 u),
%\end{align}
This equation possesses the Hamiltonian and additional nonlinear invariant
\begin{align}
\label{eq:fw-invariants-nonlinear}
  \H(u) & = \xint \left(\frac{u^3}{6} + \frac{1}{2}u \S_P^{-1} u\right) \dif x, &
  \inv(u) &= \int_{\xmin}^{\xmax} u^2 \dif x.
\end{align}
It can be written in the form \eqref{hamiltonian-form} by taking $J=-\partial_x$.
The discrete equivalent $\vec{u}^T \M \vec{u}$
of $\inv(u)$ is conserved by semidiscretizations of the form
\begin{equation}
\label{eq:fw-inv-periodic-SBP}
  \partial_t \vec{u}
  + \frac{1}{3} \D1 \vec{u}^2
  + \frac{1}{3} \mat{u} \D1 \vec{u}
  + (\I - \D2)^{-1} \D1 \vec{u}
  =
  \vec{0}
\end{equation}
where $\D1$ and $\D2$ commute \cite{ranocha2021broad}.

\subsection{Camassa-Holm equation}
\label{sec:ch}

The Camassa-Holm equation \cite{camassa1993integrable}
\begin{equation}
\label{eq:ch-dir}
\begin{aligned}
  \S \partial_t u(t,x)
  + \partial_x \biggl(
    \frac{3}{2} u(t,x)^2
    - \frac{1}{2} (\partial_x u(t,x))^2
    - u(t,x) \partial_x^2 u(t,x)
  \biggr)
  &= 0,
%   \\& t \in (0, T), x \in (\xmin, \xmax),
  \\
  u(0, x) &= u^0(x),
%   \\& x \in [\xmin, \xmax],
\end{aligned}
\end{equation}
with periodic boundary conditions can also be written as
\begin{multline}
\label{eq:ch-inv}
  \partial_t u
  + S_P^{-1} \bigl(
    \partial_x u^2
    + u \partial_x u
    - \alpha \partial_x (u \partial_x^2 u)
    - (1 - \alpha) \partial_x^2 (u \partial_x u)
    \\
    - (2 \alpha - 1) (\partial_x u) (\partial_x^2 u)
  \bigr)
  = 0,
\end{multline}
where $\alpha \in \R$ is a parameter determining the split form \cite{ranocha2021broad}.
%In addition to the mass $\int u$,
This equation possesses the Hamiltonian and an additional nonlinear invariant
\begin{align}
\label{eq:ch-invariants}
  \H(u) & = \frac{1}{2} \xint \left(u^3 + u (\partial_x u)^2\right) \dif x, &
  \inv(u) &= \frac{1}{2} \int_{\xmin}^{\xmax} \bigl( u^2 + (\partial_x u)^2 \bigr) \dif x.
\end{align}
It can be written in the form \eqref{hamiltonian-form} by taking $J=-\frac{1}{2}\S_P^{-1}\partial_x$.
The discrete equivalent
$\frac{1}{2} \vec{u}^T \M (\I - \D2) \vec{u}$
of $\inv(u)$ is conserved by semidiscretizations of the form \cite{ranocha2021broad}
\begin{equation}
\label{eq:ch-inv-periodic-SBP}
  \partial_t \vec{u}
  + (\I - \D2)^{-1} \bigl(
    \D1 \vec{u}^2
    + \vec{u} \D1 \vec{u}
    - \frac{1}{2} \D1 (\vec{u} \D2 \vec{u})
    - \frac{1}{2} \D2 (\vec{u} \D1 \vec{u})
  \bigr)
  =
  \vec{0}.
\end{equation}

\subsection{Degasperis-Procesi equation}
\label{sec:dp}

The Degasperis-Procesi equation \cite{degasperis2002new}
\begin{equation}
\label{eq:dp-dir}
\begin{aligned}
  \S \partial_t u(t,x)
  + (4 \I - \partial_x^2) \partial_x f(u(t,x))
  &= 0,
  %&& t \in (0, T), x \in (\xmin, \xmax),
  \\
  u(0, x) &= u^0(x),
  %&& x \in [\xmin, \xmax],
  \\
  f(u) &= \frac{u^2}{2},
\end{aligned}
\end{equation}
with periodic boundary conditions can also be written as
\begin{equation}
\label{eq:dp-inv}
  \partial_t u(t,x)
  + S_P^{-1} (4\I - \partial_x^2) \partial_x f(u(t,x))
  = 0,
\end{equation}
where $S_P^{-1}$ is the inverse of the elliptic
operator $\I - \partial_x^2$ with periodic boundary conditions.
%In addition to the linear functional
%\begin{align}
%  \label{eq:dp-invariants-linear}
%  \int_{\xmin}^{\xmax} ( u - \partial_x^2 u ),
%\end{align}
This equation admits (among others) the nonlinear invariants \cite{degasperis2002new}
\begin{align}
  \label{eq:dp-invariants-nonlinear}
  \H(u) & = -\xint \frac{1}{6} u^3 \dif x, &
  \inv(u) &= \frac{1}{2} \int_{\xmin}^{\xmax} \bigl( (u - \partial_x^2 u) v \bigr) \bigr) \dif x,
  && v = (4\I - \partial_x^2)^{-1} u.
\end{align}
It can be written in the form \eqref{hamiltonian-form} by taking
$J = S_P^{-1} (4\I - \partial_x^2) \partial_x$.
The discrete equivalent
$\frac{1}{2} \vec{u}^T (\I - \D2)^T \M (4 \I - \D2)^{-1} \vec{u}$
of $\inv(u)$ is conserved by semidiscretizations of the form \cite{ranocha2021broad}
\begin{equation}
\label{eq:dp-inv-periodic-SBP}
  \partial_t \vec{u}
  + \frac{1}{3} (\I - \D2)^{-1} (4\I - \D2) \left(
    \D1 \vec{u}^2 + \vec{u} \D1 \vec{u}
  \right)
  =
  \vec{0}.
\end{equation}

\subsection{BBM-BBM system}
\label{sec:bbm_bbm}

The BBM-BBM system
\cite{bona2002boussinesq,bona2004boussinesq,antonopoulos2009initial,antonopoulos2010numerical}
\begin{equation}
\label{eq:bbm_bbm-dir}
\begin{aligned}
  \partial_t \eta(t,x)
  + \partial_x u(t,x)
  + \partial_x \bigl( \eta(t,x) u(t,x) \bigr)
  - \partial_t \partial_x^2 \eta(t,x)
  &= 0,
%   && t \in (0, T), x \in (\xmin, \xmax),
  \\
  \partial_t u(t,x)
  + \partial_x \eta(t,x)
  + \partial_x \frac{u(t,x)^2}{2}
  - \partial_t \partial_x^2 u(t,x)
  &= 0,
%   && t \in (0, T), x \in (\xmin, \xmax),
  \\
  \eta(0, x) &= \eta^0(x),
%   && x \in [\xmin, \xmax],
  \\
  u(0, x) &= u^0(x),
%   && x \in [\xmin, \xmax],
\end{aligned}
\end{equation}
with periodic boundary conditions can also be written as
\begin{equation}
\label{eq:bbm_bbm-inv}
\begin{aligned}
  \partial_t \eta(t,x)
  + S_P^{-1} \partial_x \bigl( u(t,x) + \eta(t,x) u(t,x) \bigr)
  &= 0,
  \\
  \partial_t u(t,x)
  + S_P^{-1} \partial_x \biggl( \eta(t,x) + \frac{u(t,x)^2}{2} \biggr)
  &= 0.
\end{aligned}
\end{equation}
This equation admits the nonlinear invariants
\begin{align}
\label{eq:bbm_bbm-invariants}
  \H(\eta, u) &= -\frac{1}{2}\int_{\xmin}^{\xmax} (\eta^2 + (1 + \eta) u^2) \dif x, &
  \inv(\eta, u) &= \int_{\xmin}^{\xmax} (\eta u + (\partial_x \eta)(\partial_x u)) \dif x,
\end{align}
and can be written in the form \eqref{hamiltonian-form} by taking
\begin{align*}
    J &= \begin{pmatrix} 0 & S^{-1}\partial_x  \\ S^{-1}\partial_x  & 0 \end{pmatrix}.
\end{align*}
As mentioned in Section~\ref{sec:rel-eq}, we expect to observe linear
error growth whenever the numerical method conserves \emph{either}
$\H(\eta, u)$ or $\inv(\eta, u)$.  For this example, we demonstrate this by testing
methods that conserve each of these quantities separately.

The discrete equivalent
$-\frac{1}{2} (\vec{\eta}^T M \vec{\eta} + \vec{u}^T M (\vec{1} + \vec{\eta}) \vec{u})$
of $\H(\eta, u)$ is conserved by semidiscretizations of the form
\begin{equation}
\label{eq:bbm_bbm-inv-SBP}
\begin{aligned}
  \partial_t \vec{\eta}
  + (\I - \D2)^{-1} \D1 ( \vec{u} + \vec{\eta} \vec{u} )
  &= \vec{0},
  \\
  \partial_t \vec{u}
  + (\I - \D2)^{-1} \D1 \biggl( \vec{\eta} + \frac{1}{2} \vec{u}^2 \biggr)
  &= \vec{0}.
\end{aligned}
\end{equation}
where $\D1, \D2$ commute \cite{ranocha2021broad}.

Conservation of the discrete equivalent
$\vec{u}^T M (\I - \D2) \vec{\eta}$
of $\inv(\eta, u)$ can be achieved with the standard Galerkin
method \cite{mitsotakis2021conservative} or SBP methods using the split form
\begin{equation}
\label{eq:bbm_bbm-inv-SBP-quadratic}
\begin{aligned}
  \partial_t \vec{\eta}
  + (\I - \D2)^{-1} \D1 ( \vec{u} + \vec{\eta} \vec{u} )
  &= \vec{0},
  \\
  \partial_t \vec{u}
  + (\I - \D2)^{-1} (\D1 \vec{\eta} + \vec{u} \D1 \vec{u})
  &= \vec{0};
\end{aligned}
\end{equation}
see Theorem~\ref{thm:bbm_bbm-inv-SBP-quadratic} below.

\begin{theorem}
\label{thm:bbm_bbm-inv-SBP-quadratic}
  If $\D1$ is a periodic first-derivative SBP operator
  and $\D2$ is  a periodic second-derivative SBP operator with diagonal mass matrix,
  then the semidiscretization \eqref{eq:bbm_bbm-inv-SBP-quadratic}
  conserves the quadratic invariant $\inv$ (defined in
  \eqref{eq:bbm_bbm-invariants}).
\end{theorem}
\begin{proof}
  Using $\vec{1}^T \M (\I - \D2)^{-1} = \vec{1}^T \M$
  \cite[Lemma~2.28]{ranocha2021broad} and $\vec{1}^T \M \D1 = \vec{0}^T$
  \cite[Lemma~2.27]{ranocha2021broad} results in
  \begin{equation}
    \vec{1}^T \M \partial_t \vec{\eta}
    =
    - \vec{1}^T \M (\I - \D2)^{-1} \D1 ( \vec{u} + \vec{\eta} \vec{u} )
    =
    0.
  \end{equation}
  Applying the SBP property \eqref{eq:D1-periodic} additionally and using that
  $\M$ is diagonal yields
  \begin{equation}
    \vec{1}^T \M \partial_t \vec{u}
    =
    - \vec{1}^T \M (\I - \D2)^{-1} ( \D1 \vec{\eta} + \vec{u} \D1 \vec{u} )
    =
    - \vec{u}^T \M \D1 \vec{u}
    = 0.
  \end{equation}
  Finally, similar arguments based on the symmetry of $\I - \D2$ with respect
  to the mass matrix $\M$ lead to
  \begin{equation}
  \begin{aligned}
    \partial_t \inv(\vec{\eta}, \vec{u})
    &=
      \vec{u}^T \M (\I - \D2) \partial_t \vec{\eta}
    + \vec{\eta}^T \M (\I - \D2) \partial_t \vec{u}
    \\
    &=
    - \vec{u}^T \M \D1 ( \vec{u} + \vec{\eta} \vec{u} )
    - \vec{\eta}^T \M ( \D1 \vec{\eta} + \vec{u} \D1 \vec{u} )
    =
    0.
  \end{aligned}
    \qedhere
  \end{equation}
\end{proof}

\subsection{Holm-Hone equation}
\label{sec:hh}

As our last example in this section, we consider
the Holm-Hone equation \cite{holm2003nonintegrability}:
\begin{equation}
\label{eq:hh-dir}
\begin{aligned}
  (4 - 5 \partial_x^2 + \partial_x^4) \partial_t u
  + u \partial_x^5 u
  + 2 (\partial_x u) \partial_x^4 u
  - 5 u \partial_x^3 u
  - 10 (\partial_x u) \partial_x^2 u
  + 12 u \partial_x u
  &= 0,
%   \\& t \in (0, T), x \in (\xmin, \xmax),
  \\
  u(0, x) &= u^0(x),
%   \\& x \in [\xmin, \xmax],
\end{aligned}
\end{equation}
with periodic boundary conditions.  It can also be written as
\begin{equation}
\label{eq:hh-inv}
  \partial_t u
  + (4\I - 5 \partial_{x,P}^2 + \partial_{x,P}^4)^{-1} \left(
      \partial_x \bigl( u (4 \I - 5 \partial_x^2 + \partial_x^4) u \bigr)
    + (\partial_x u) (4 \I - 5 \partial_x^2 + \partial_x^4) u
  \right)
  = 0,
\end{equation}
where $(4\I - 5 \partial_{x,P}^2 + \partial_{x,P}^4)^{-1}$ is the
inverse of the elliptic operator $4\I - 5 \partial_x^2 + \partial_x^4$
with periodic boundary conditions.
%The functionals
%\begin{subequations}
%\label{eq:hh-invariants}
%\begin{align}
%  \label{eq:hh-invariants-mass}
%  J^{\text{HH}}_1(u)
%  &= \int_{\xmin}^{\xmax} u,
%  \\
%  \label{eq:hh-invariants-linear}
%  J^{\text{HH}}_2(u)
%  &= \int_{\xmin}^{\xmax} (4\I - \partial_x^2) (\I - \partial_x^2) u
%  = \int_{\xmin}^{\xmax} (4\I - 5 \partial_x^2 + \partial_x^4) u,
%  \\
This equation can be written in the form \eqref{hamiltonian-form} using the
nonlinear invariant
\cite{holm2003nonintegrability,wang2019symmetry}
\begin{align}
  \label{eq:hh-invariants-quadratic}
  \H(u)
  &= \frac{1}{2} \int_{\xmin}^{\xmax} u (4\I - \partial_x^2) (\I - \partial_x^2) u
  = \frac{1}{2} \int_{\xmin}^{\xmax} \bigl( 4 u^2 + 5 (\partial_x u)^2 + (\partial_x^2 u)^2 \bigr).
\end{align}
Unlike the previous equations, this one appears to have no second nonlinear
invariant; instead the translation symmetry is generated by the linear invariant
\begin{align}
    \inv(u) & = u - \partial_x^2 u.
\end{align}
The relevant operator appearing in \eqref{hamiltonian-form} is
\begin{align*}
    J & = - (\inv \partial_x + \partial_x \inv).
\end{align*}
Both the discrete equivalent
$\vec{1}^T M (\I - \D2) \vec{u}$
of the linear invariant $\inv(u)$ and the discrete equivalent
$\frac{1}{2} \vec{u}^T M (4 \I - 5 \D2 + \D4) \vec{u}$
of the nonlinear invariant $\H(u)$ are
conserved by semidiscretizations of the form
\begin{equation}
\label{eq:hh-inv-periodic-SBP}
\begin{split}
  \partial_t \vec{u}
  =
  - (4\I - 5 \D2 + \D4)^{-1} \Bigl(
      \D1 \bigl( \vec{u} (4\I - 5 \D2 + \D4) \vec{u} \bigr)
    + (\D1 \vec{u}) (4\I - 5 \D2 + \D4) \vec{u}
  \Bigr)
\end{split}
\end{equation}
where $\D1, \D2, \D4$ commute \cite{ranocha2021broad}.

\subsection{Numerical results}
For each of the five equations just studied, we perform a similar numerical
experiment.
We use a numerically generated smooth solitary wave solution as initial condition in
a periodic domain, and apply the corresponding SBP spatial discretization with
Fourier collocation methods in space and Tsitouras' fifth-order
Runge-Kutta method \cite{tsitouras2011runge} in time, with or without using
relaxation to enforce conservation of a nonlinear invariant $\inv$.
Specific parameters for each simulation are shown in Table~\ref{tbl:params}.

\begin{table}[!htp]
\begin{center}
\caption{Summary of numerical simulations. \#DOF denotes the number of the
         degrees of freedom (number of grid points) of the spatial
         semidiscretization. Tol. is the local error tolerance used for adaptive
         time stepping.}
\label{tbl:params}
\begin{tabular*}{\linewidth}{@{\extracolsep{\fill}}cccccc@{}}
  \toprule
  Equation & Discretization & Domain & \#DOF & Tol. & Fig. \\
  \midrule
  Fornberg-Whitham \eqref{eq:fw-dir} & \eqref{eq:fw-inv-periodic-SBP} & $[-80,80]$ & $2^8$ & 1e-5 & \ref{fig:fw-error} \\
  Camassa-Holm \eqref{eq:ch-dir} & \eqref{eq:ch-inv-periodic-SBP} & $[-40,40]$ & $2^9$ & 1e-7 & \ref{fig:ch-error} \\
  Degasperis-Procesi \eqref{eq:dp-dir} & \eqref{eq:dp-inv-periodic-SBP} & $[-40,40]$ & $2^8$ & 1e-5 & \ref{fig:dp-error} \\
  BBM-BBM \eqref{eq:bbm_bbm-dir} & \eqref{eq:bbm_bbm-inv-SBP} & $[-40,40]$ & $2^8$ & 1e-5 & \ref{fig:bbm_bbm-error} \\
  %BBM-BBM \eqref{eq:bbm_bbm-dir} & \eqref{eq:bbm_bbm-inv-SBP-quadratic} & $[-40,40]$ & $2^8$ & 1e-5 & \ref{fig:dp-error} \\
  Holm-Hone \eqref{eq:hh-dir} & \eqref{eq:hh-inv-periodic-SBP} & $[-40,40]$ & $2^9$ & 1e-9 & \ref{fig:hh-error}
  \\
  \bottomrule
\end{tabular*}
\end{center}
\end{table}

\begin{figure}[htb]
\centering
  \begin{subfigure}{0.4\textwidth}
  \centering
    \includegraphics[width=\textwidth]{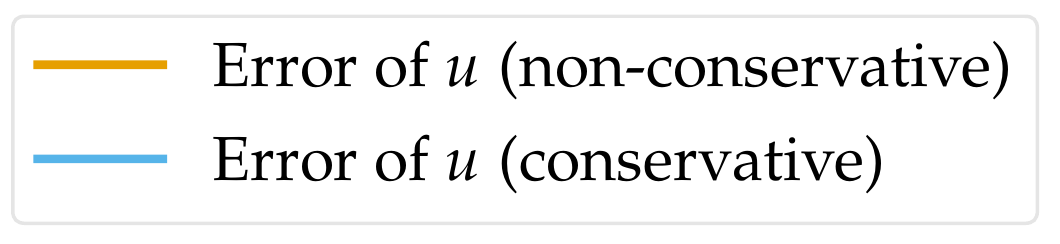}
  \end{subfigure}%
  \\
  \begin{subfigure}{0.49\textwidth}
  \centering
    \includegraphics[width=\textwidth]{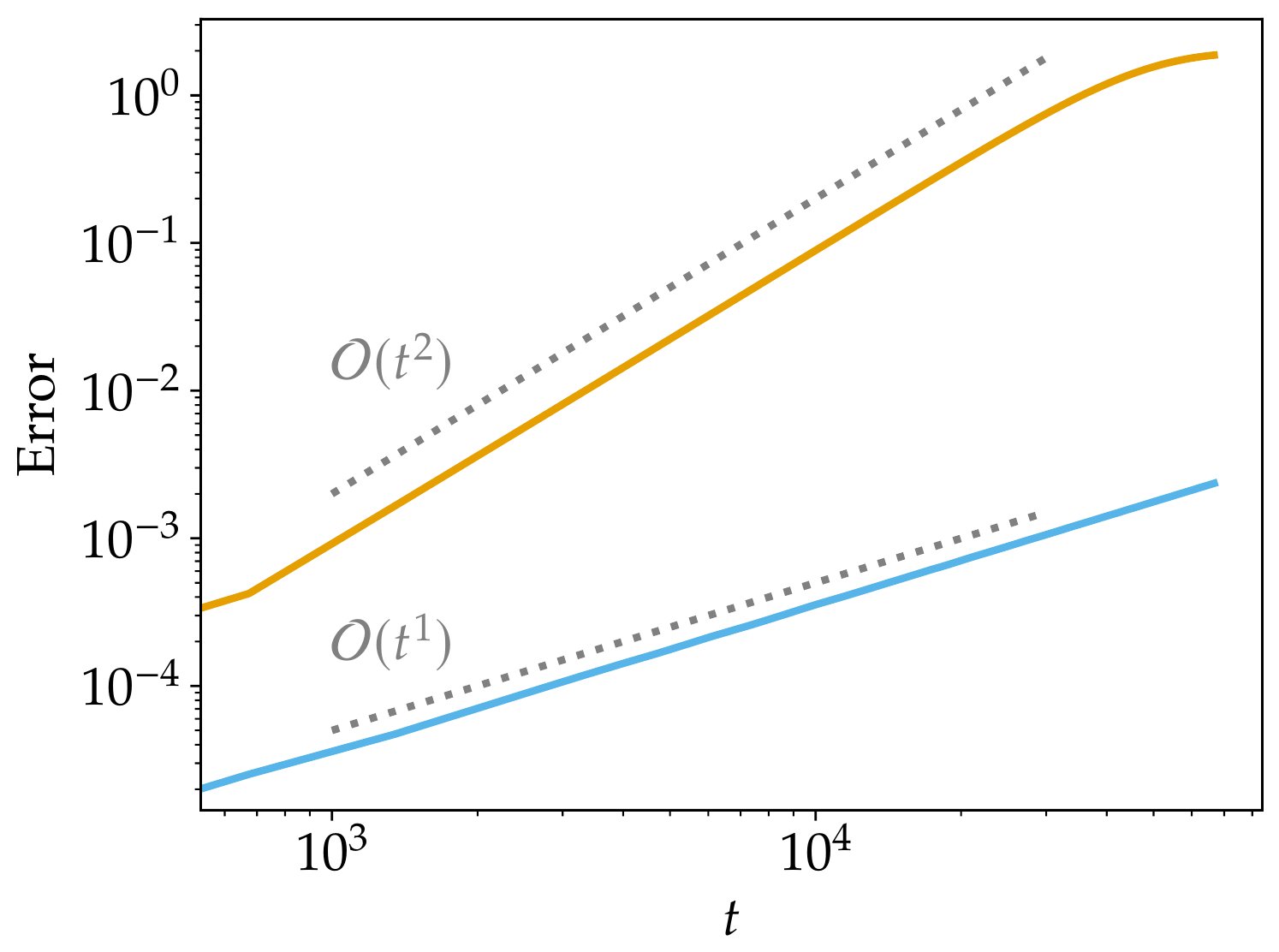}
    \caption{Fornberg-Whitham equation \eqref{eq:fw-dir}.}
    \label{fig:fw-error}
  \end{subfigure}%
  \hfill
  \begin{subfigure}{0.49\textwidth}
  \centering
    \includegraphics[width=\textwidth]{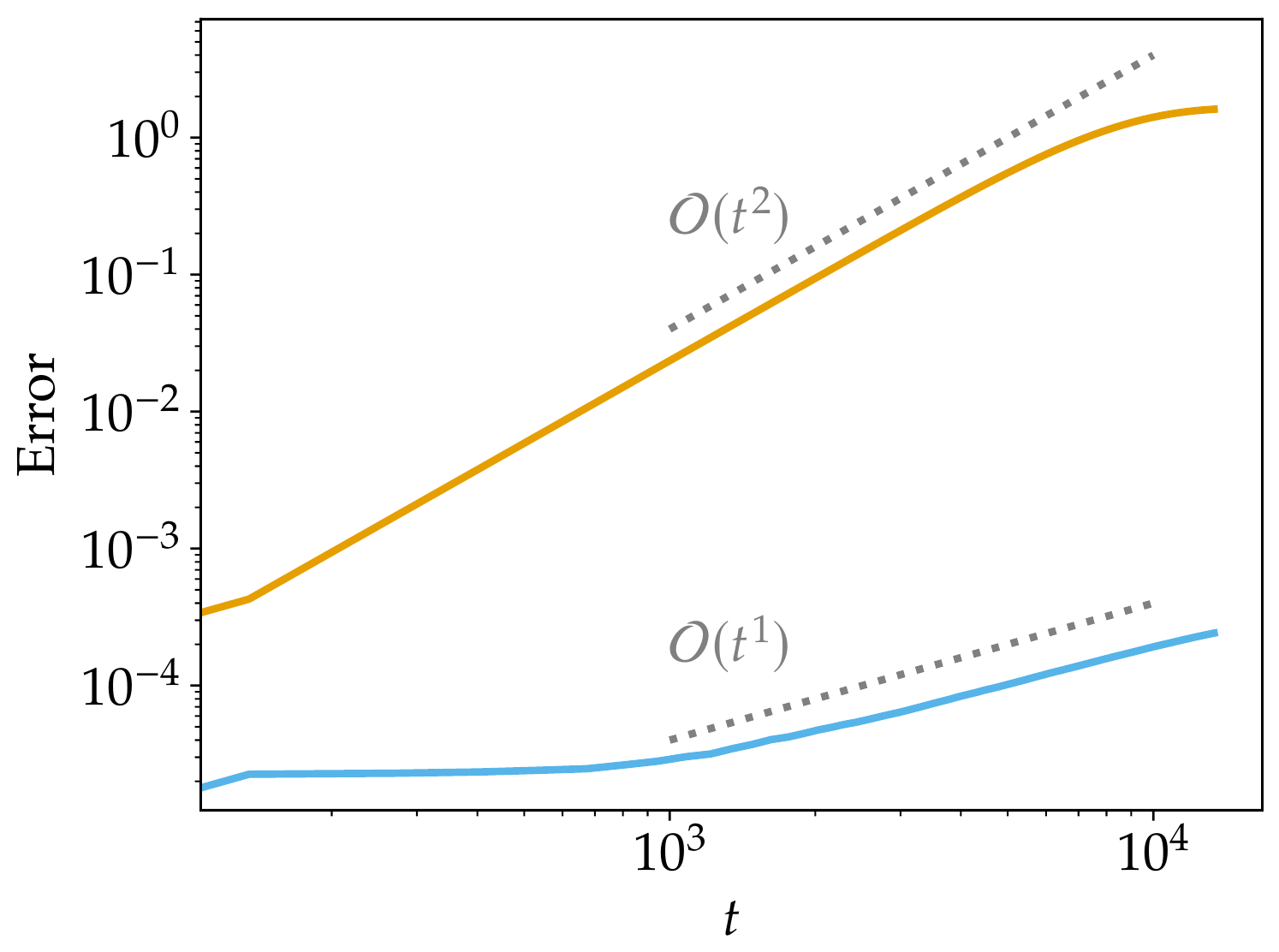}
    \caption{Camassa-Holm equation \eqref{eq:ch-dir}.}
    \label{fig:ch-error}
  \end{subfigure}%
  \\
  \begin{subfigure}{0.49\textwidth}
  \centering
    \includegraphics[width=\textwidth]{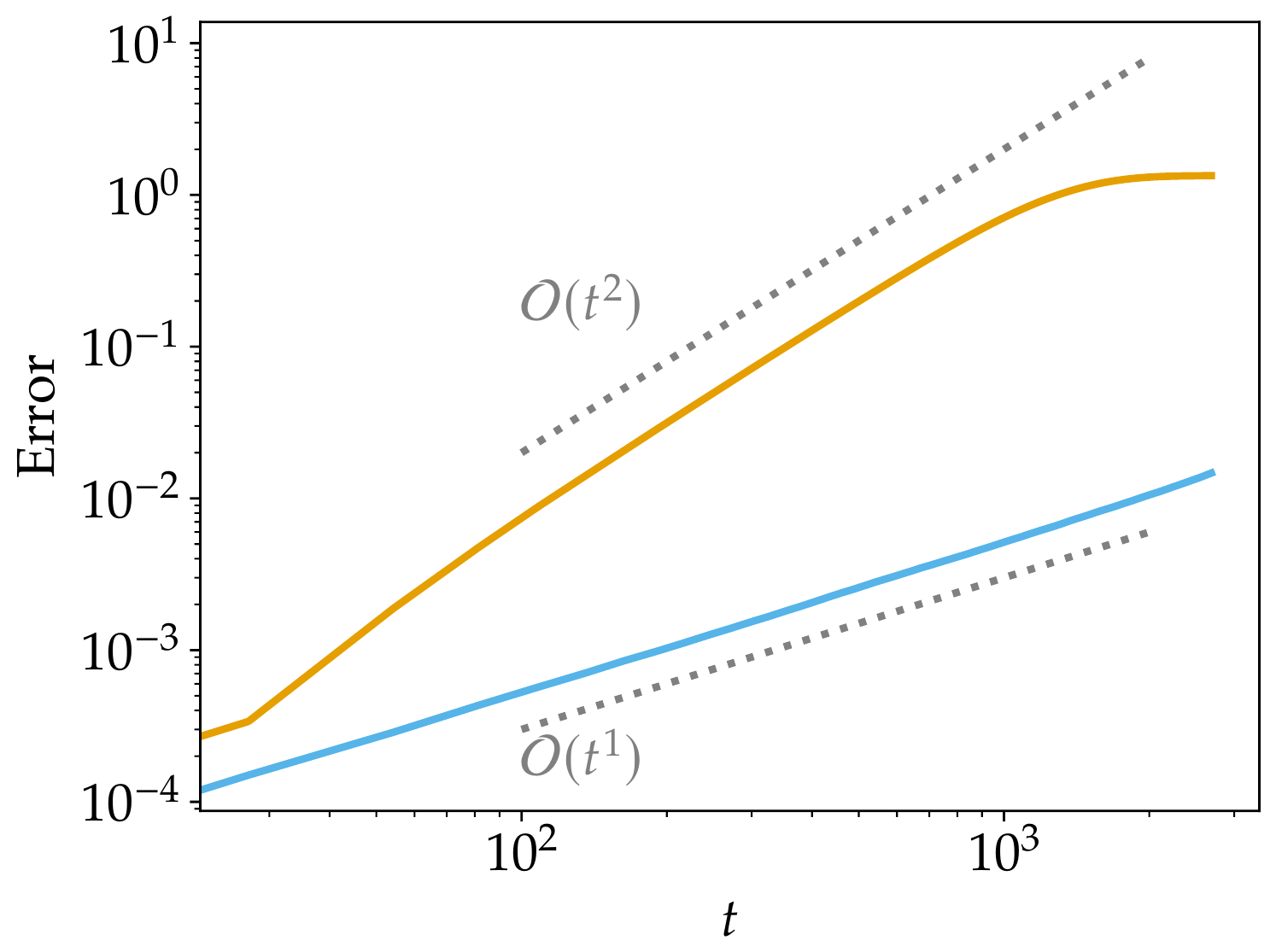}
    \caption{Degasperis-Procesi equation \eqref{eq:dp-dir}.}
    \label{fig:dp-error}
  \end{subfigure}%
  \hfill
  \begin{subfigure}{0.49\textwidth}
  \centering
    \includegraphics[width=\textwidth]{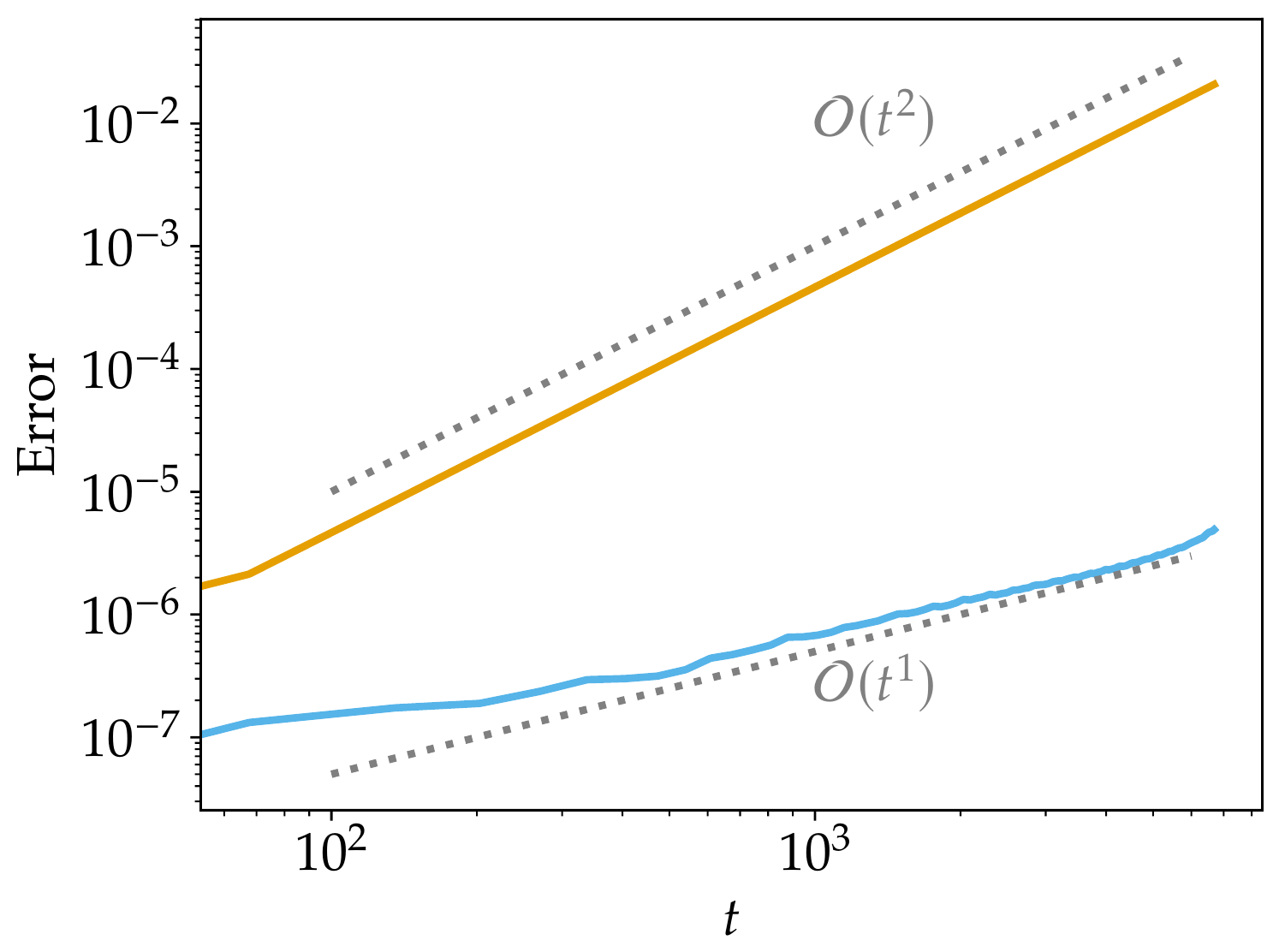}
    \caption{Holm-Hone equation \eqref{eq:hh-dir}.}
    \label{fig:hh-error}
  \end{subfigure}%
  \\
  \caption{Error growth in time for numerically generated solitary wave solutions
           using Fourier collocation in space and 5th-order Runge-Kutta in time.
           The conservative method uses relaxation to enforce conservation of
           a nonlinear invariant.}
  \label{fig:fw-ch-dp-hh-error}
\end{figure}

\begin{figure}[htb]
\centering
  \includegraphics[width=0.9\textwidth]{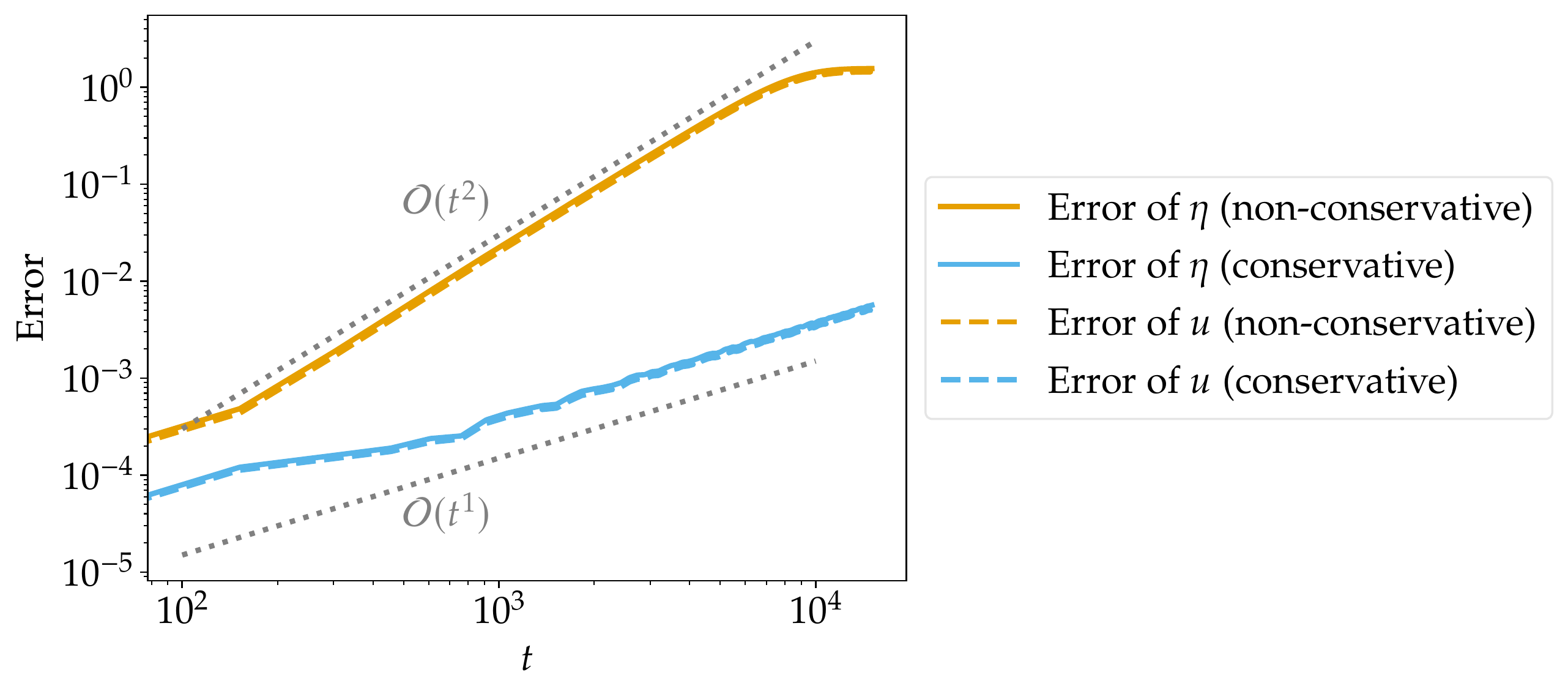}
  \caption{Error growth in time for a numerically generated solitary wave solution
           of the BBM-BBM system \eqref{eq:bbm_bbm-dir}, using \eqref{eq:bbm_bbm-inv-SBP}
           with Fourier collocation in space and 5th-order Runge-Kutta in time.
           The conservative method uses relaxation to enforce conservation of $\H$ (see \eqref{eq:bbm_bbm-invariants}).}
  \label{fig:bbm_bbm-error}
\end{figure}

\begin{figure}[htb]
\centering
  \begin{subfigure}{0.6\textwidth}
    \includegraphics[width=\textwidth]{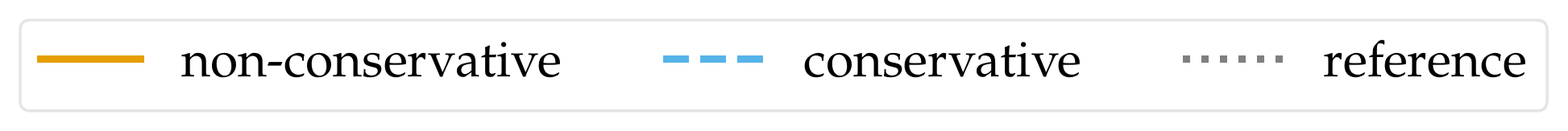}
  \end{subfigure}%
  \\
  \begin{subfigure}{0.49\textwidth}
    \includegraphics[width=\textwidth]{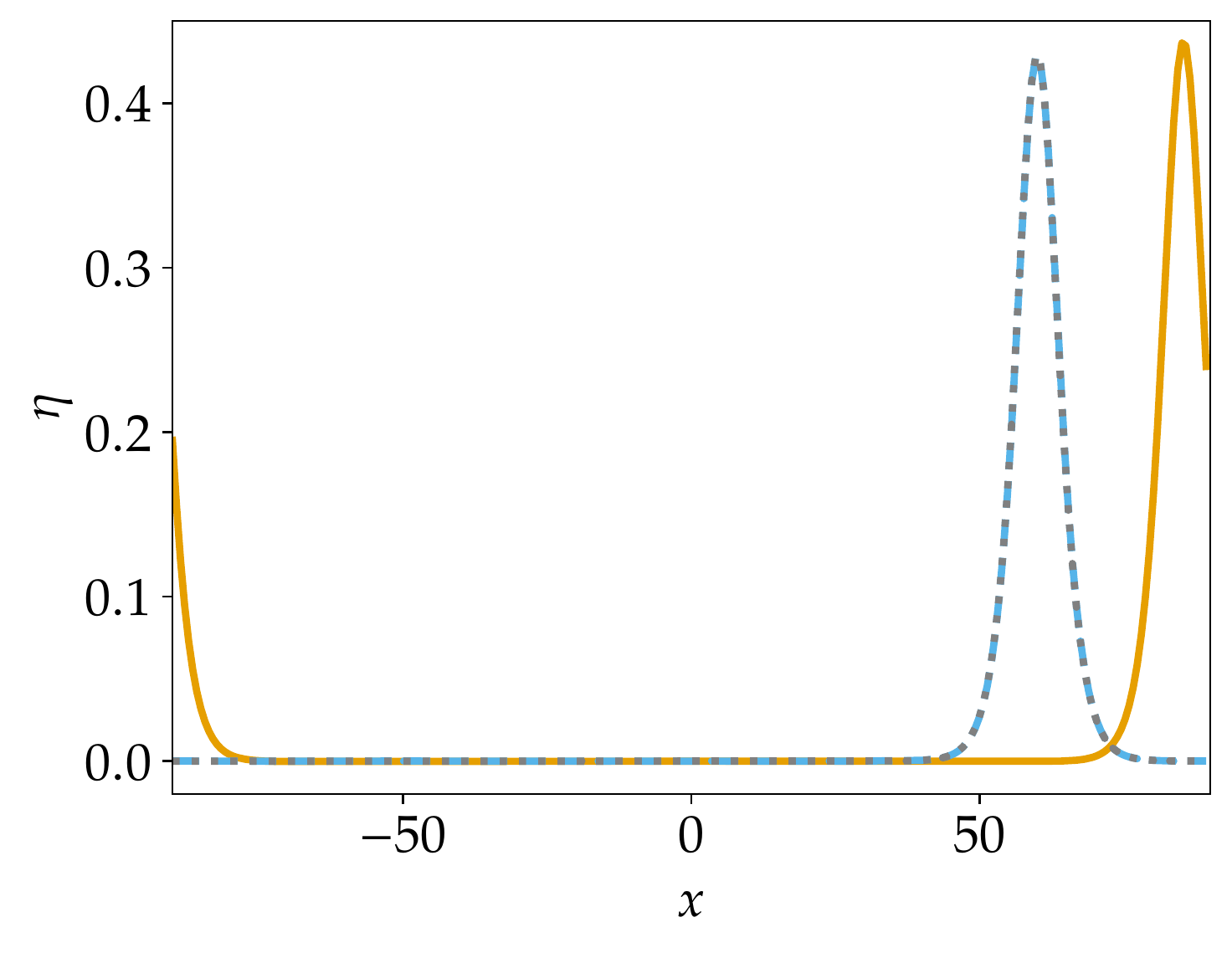}
    \caption{Numerical solution of $\eta$.}
  \end{subfigure}%
  \hspace*{\fill}
  \begin{subfigure}{0.49\textwidth}
    \includegraphics[width=\textwidth]{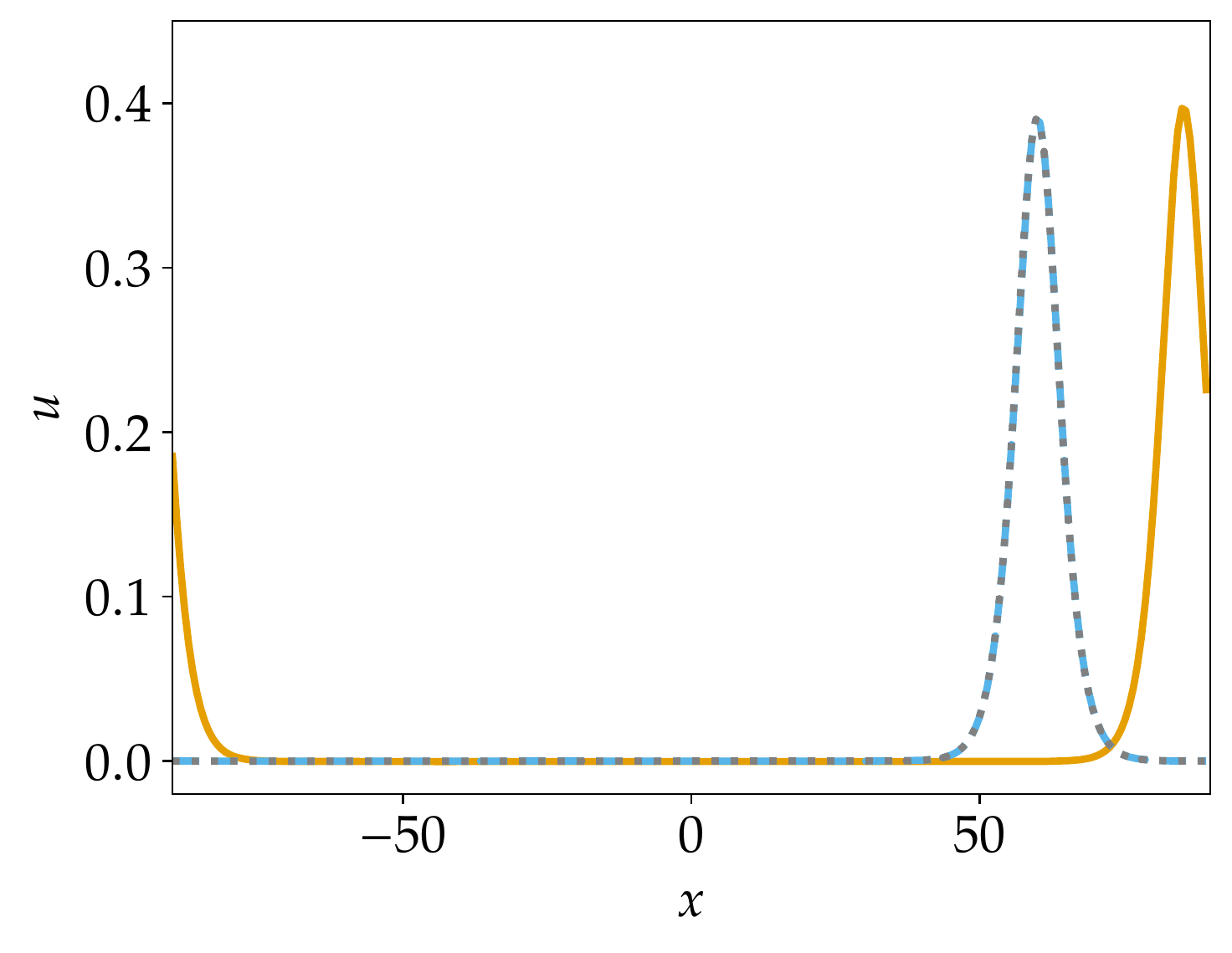}
    \caption{Numerical solution of $u$.}
  \end{subfigure}%
  \caption{Numerical solutions at the final time for a numerically
           generated solitary wave solution of the BBM-BBM system
           \eqref{eq:bbm_bbm-dir} using relaxation methods conserving $\H$
           (see \eqref{eq:bbm_bbm-invariants}).}
  \label{fig:bbm_bbm-solution}
\end{figure}

In each case, we observe approximately linear error growth for the conservative
method and approximately quadratic growth for the non-conservative method, until
eventually the error for the non-conservative method saturates when the exact and
numerical solutions no longer overlap.

An example of the numerical solutions themselves is shown in Figure \ref{fig:bbm_bbm-solution},
using BBM-BBM.
The numerical solutions obtained using the non-conservative method have a
visible amplitude error and a significant phase error. In contrast, the
conservative method yields numerical solutions that are visually indistinguishable
from the reference solution.  Results for the other equations are similar.

For the BBM-BBM system, we also conduct numerical experiments using the discretization
\eqref{eq:bbm_bbm-inv-SBP-quadratic} and applying relaxation to conserve
the corresponding functional.
The error growth in time using this method is visually indistinguishable
from that shown for the method \eqref{eq:bbm_bbm-inv-SBP}. As expected,
the error of the associated conservative method grows linearly in time while
the error of the non-conservative method grows quadratically.

Additionally, we perform numerical experiments with non-smooth traveling waves.
While these are available in closed form for some of the nonlinear dispersive
wave equations studied in this article, their reduced regularity makes it more
difficult to approximate them numerically well enough in space. Note that
existing analytical results about the error growth in time described in
Section~\ref{sec:rel-eq} are based on discretization only in time and assume exactness in space.
Thus, the spatial semidiscretization must be well-resolved in order for these results
to apply to fully discrete schemes.

Here, we use eighth-order accurate finite difference methods with $2^{13}$ nodes
for peakon solutions of the form $u(t, x) = c \exp(-|x - c t|)$ with $c = 1.2$
of the Camassa-Holm equation \cite{camassa1993integrable,lenells2005travelingCH}
in the periodic domain $[-35, 35]$. The fifth-order Runge-Kutta method uses a
local error tolerance of $10^{-7}$.
The results visualized in Figure~\ref{fig:ch-nonsmooth} are in accordance
with the expectations: We observe quadratic error growth for non-conservative
methods and linear error growth for conservative methods.

\begin{figure}[htb]
\centering
  \begin{subfigure}{0.6\textwidth}
    \includegraphics[width=\textwidth]{figures/bbm_bbm_solution_legend}
  \end{subfigure}%
  \\
  \begin{subfigure}{0.49\textwidth}
  \centering
    \includegraphics[width=\textwidth]{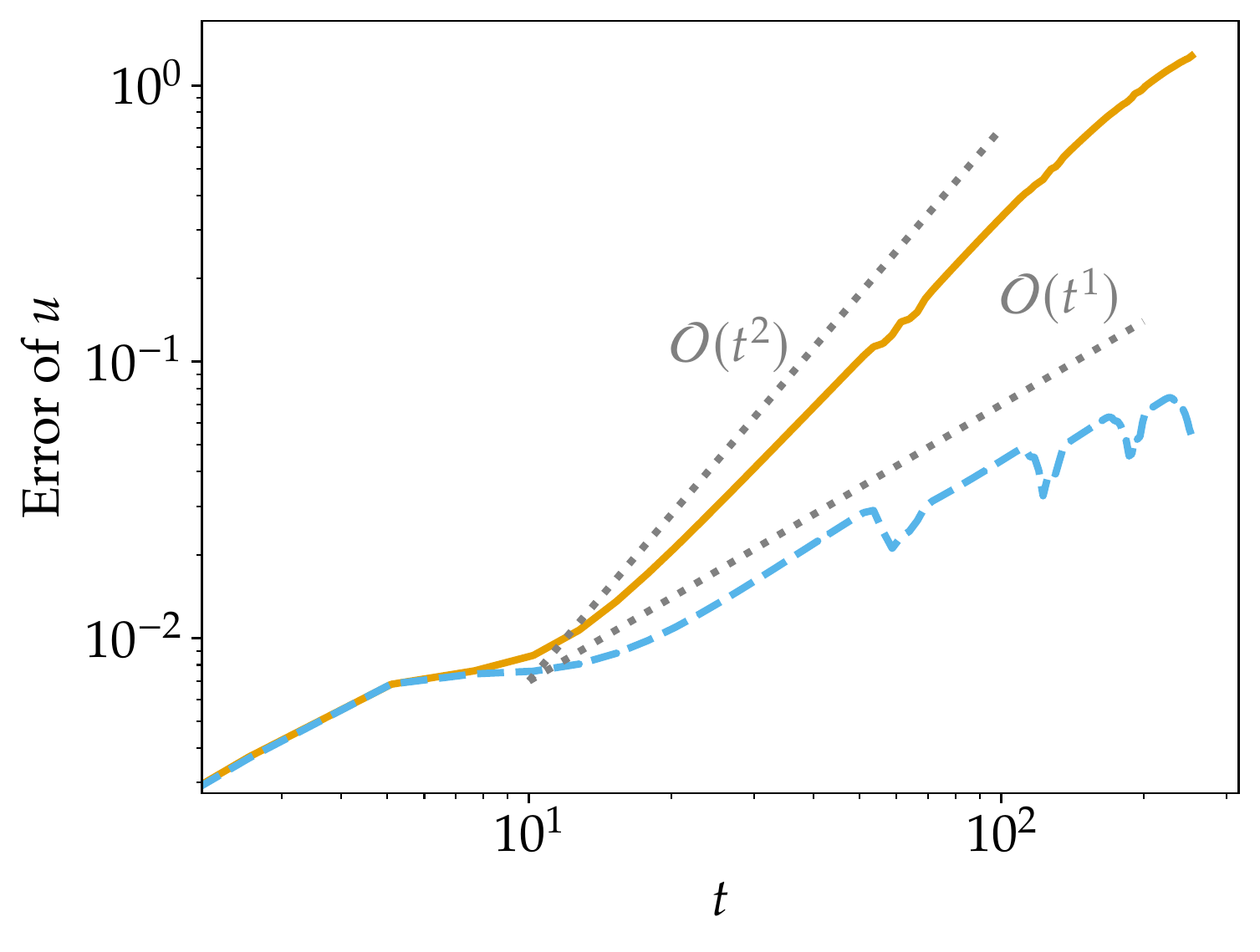}
    \caption{Error growth in time.}
  \end{subfigure}%
  \hfill
  \begin{subfigure}{0.49\textwidth}
  \centering
    \includegraphics[width=\textwidth]{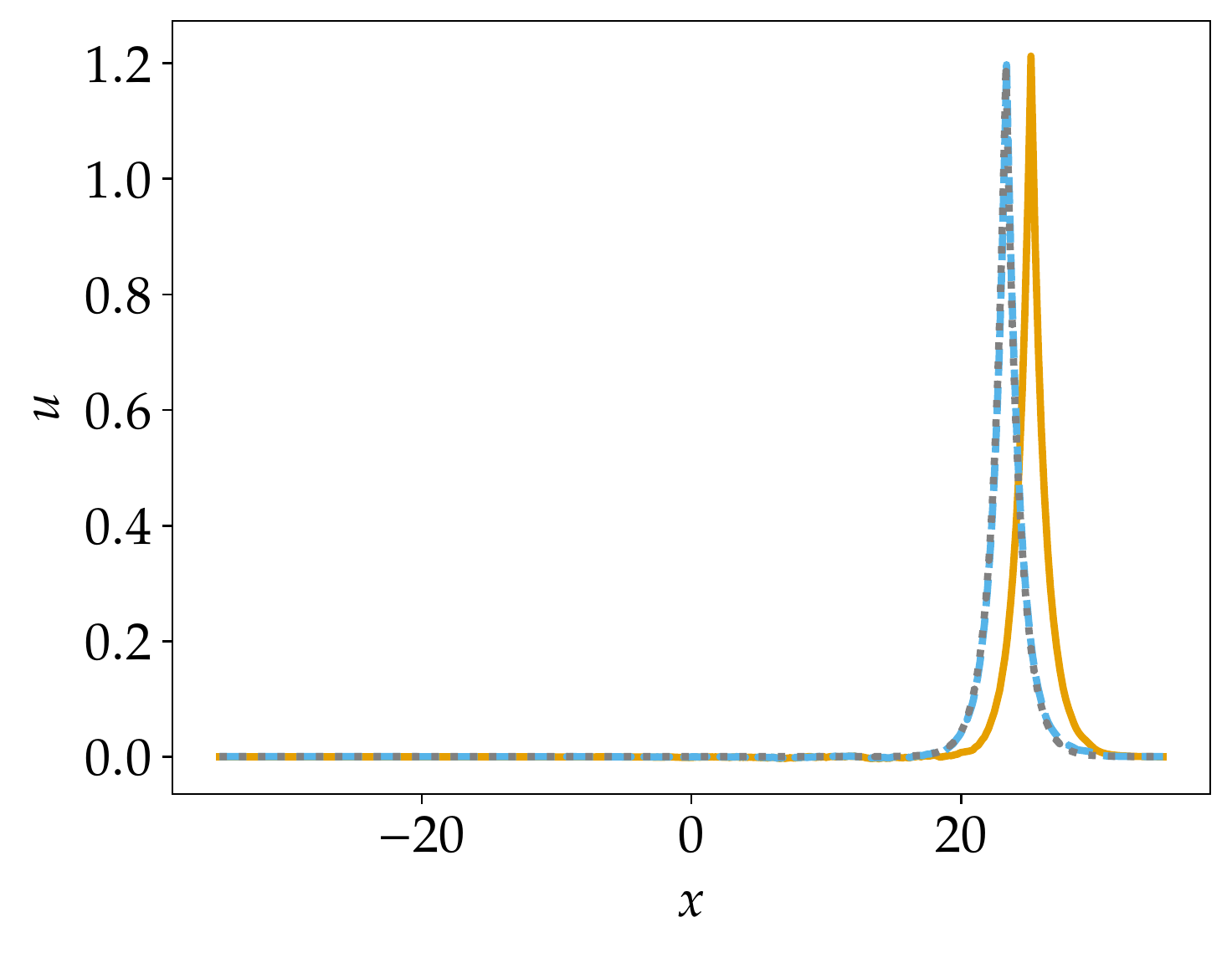}
    \caption{Numerical solutions at the final time.}
  \end{subfigure}%
  \\
  \caption{Numerical results for a non-smooth solitary wave solution of the
           Camassa-Holm equation \eqref{eq:ch-dir} using finite difference methods
           in space and 5th-order Runge-Kutta in time.
           The conservative method uses relaxation to enforce conservation of
           the nonlinear invariant $\inv(u)$.}
  \label{fig:ch-nonsmooth}
\end{figure}

\section{Significance of linear invariants}

In this section we investigate the importance of conservation of linear invariants.
First, we study the behavior of a method that preserves energy but not mass
for the KdV equation.  Then, we consider a linear PDE in which nonlinear
invariants play no role.

\subsection{Projection methods: non-conservation of mass}
\label{sec:proj}

Here we consider the propagation of a soliton solution of the KdV equation, using
three methods.  All three methods use the SBP spatial discretization of
\cite{ranocha2020relaxationHamiltonian} with eighth-order accurate finite
differences that is mass- and energy-conserving.  The time discretizations are:
\begin{itemize}
    \item An IMEX RK method of \cite{kennedy2003additive}
          that conserves mass but not energy
    \item The corresponding relaxation method
          that conserves both mass and energy
    \item The corresponding orthogonal projection method
          that conserves energy but not mass
\end{itemize}
For details on orthogonal projection, we refer to \cite[Section~IV.4]{hairer2013geometric}.
A similar comparison was conducted in \cite{ranocha2020relaxationHamiltonian};
here we conduct a more detailed study of the error growth for the projection scheme.

Results are shown in Figure \ref{fig:kdv}.  The first two methods behave
as expected, in accordance with theoretical results in \cite{frutos1997accuracy}.
The projection method initially performs similar to the relaxation method.
However, its error grows quadratically after some time and eventually leads to a
much more inaccurate solution, in accordance with the theory of
\cite{frutos1997accuracy}. Interestingly, the large global
error in total mass at late times is exhibited not primarily in a change in
mass of the soliton, but in a gradual downward drift of the ambient value of $u$,
as shown in Figure \ref{fig:kdv-zoom}.
We observed a similar behavior of orthogonal projection methods also for other
nonlinear dispersive wave equations such as the BBM equation.

\begin{figure}[htb]
\centering
  \begin{subfigure}{0.8\textwidth}
    \includegraphics[width=\textwidth]{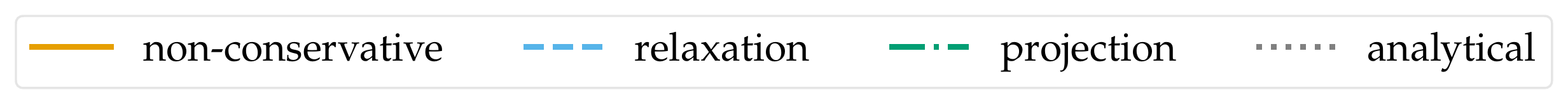}
  \end{subfigure}%
  \\
  \begin{subfigure}{0.49\textwidth}
    \includegraphics[width=\textwidth]{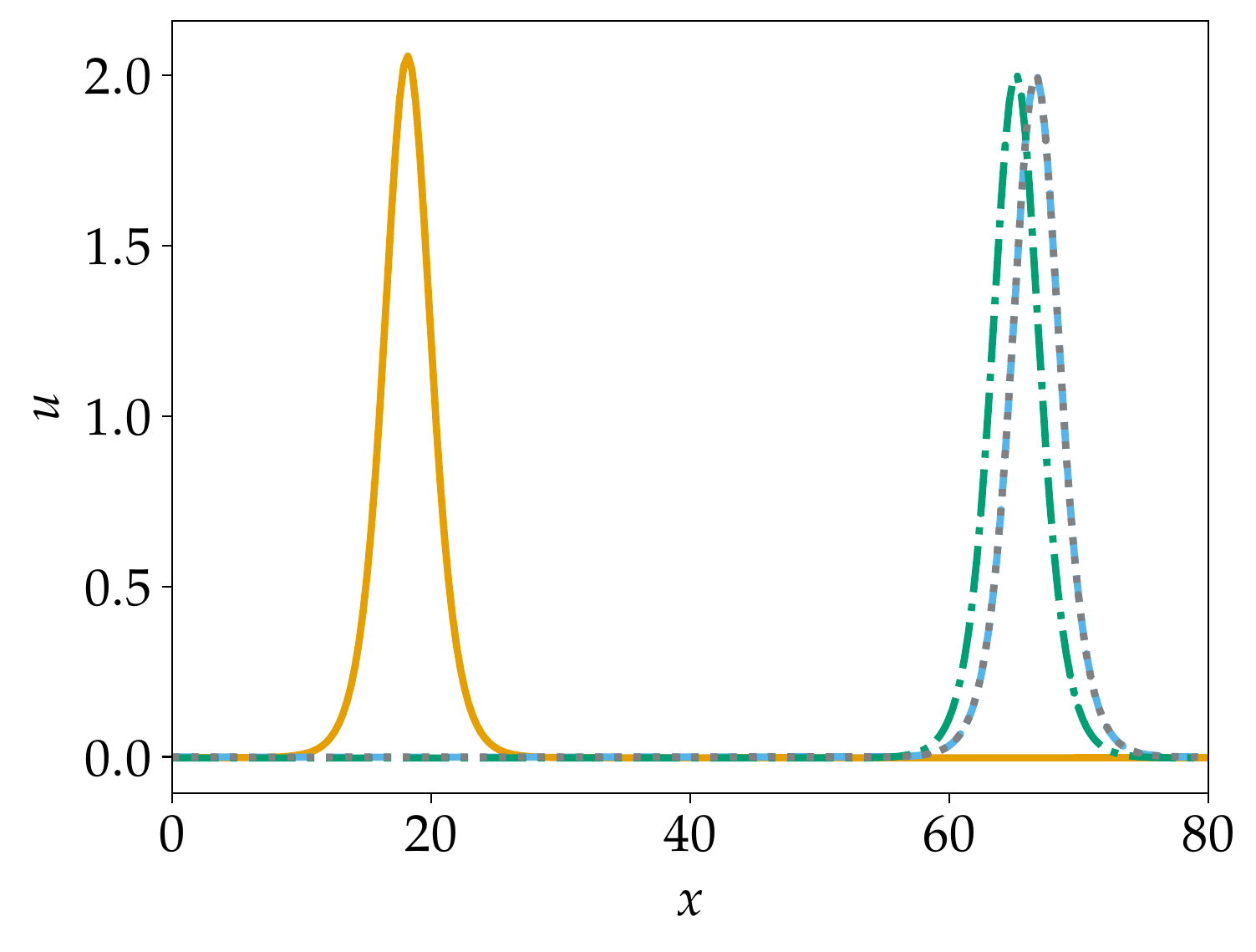}
    \caption{Numerical solution.}
  \end{subfigure}%
  \hspace*{\fill}
  \begin{subfigure}{0.49\textwidth}
    \includegraphics[width=\textwidth]{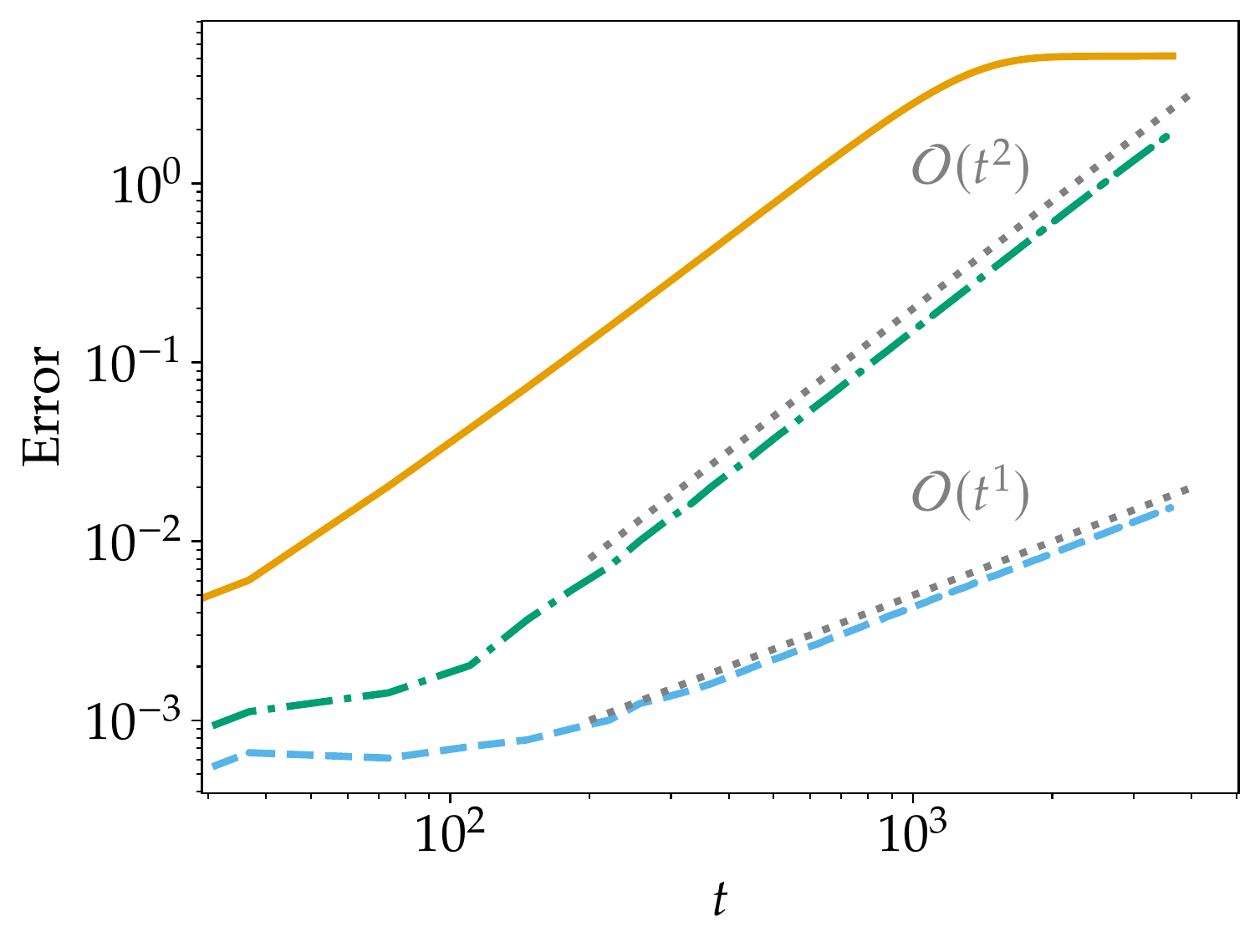}
    \caption{Error versus time.}
  \end{subfigure}%
  \caption{Numerical solutions and error over time for numerical
            propagation of a KdV soliton.}
  \label{fig:kdv}
\end{figure}

\begin{figure}[htb]
\centering
    \includegraphics[width=0.75\textwidth]{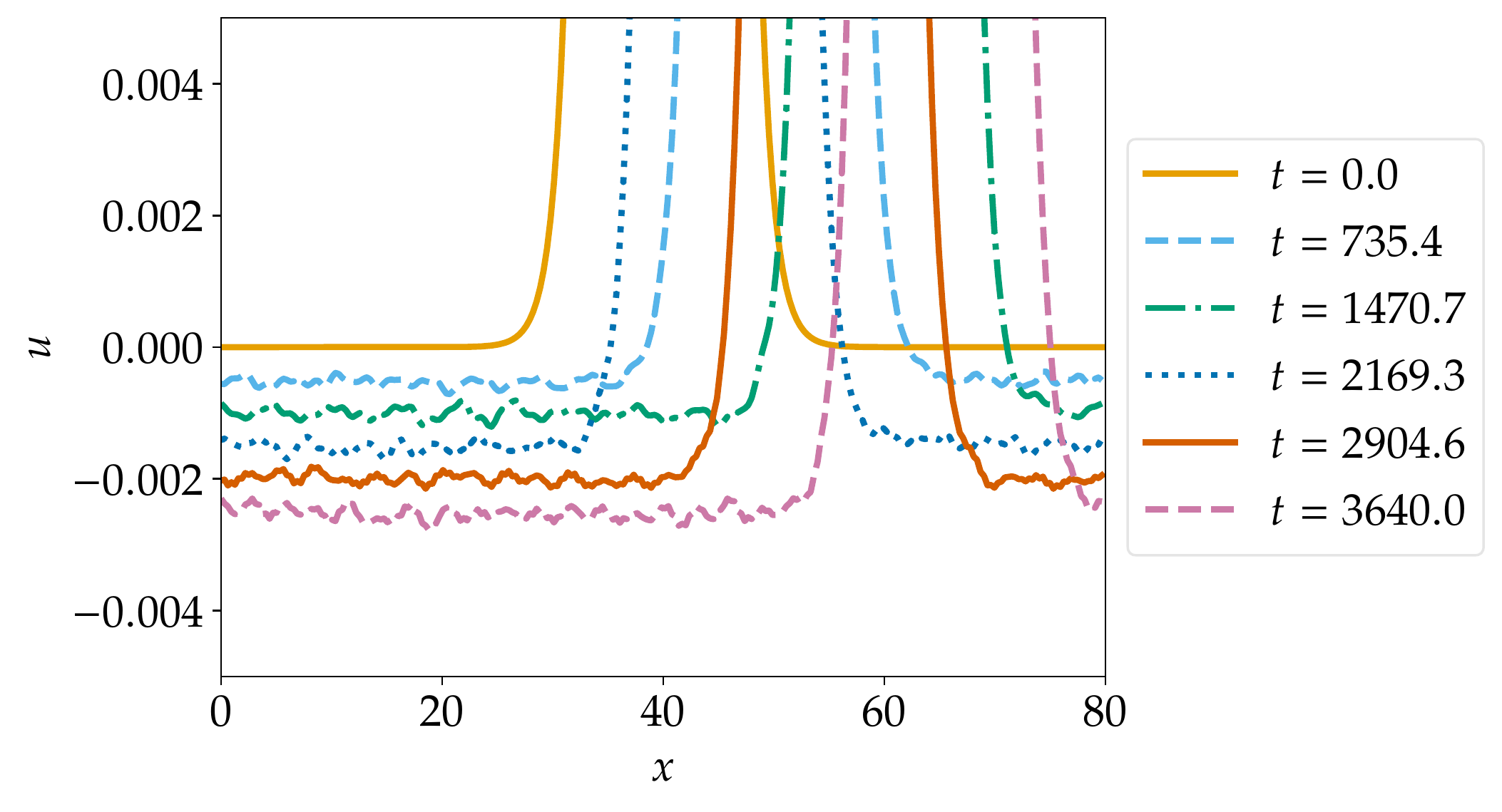}
    \caption{Closeup of numerical solution by projection at various times.}
  \label{fig:kdv-zoom}
\end{figure}

\subsection{A linear dispersive equation}
\label{sec:linear}

As alluded to in the introduction, nonlinearity is an important condition to
observe different behaviors of the error growth in time. To demonstrate that,
we consider the linear dispersive equation
\begin{equation}
\label{eq:linear}
\begin{aligned}
  (\I - \partial_x^2) \partial_t u(t,x) + \partial_x u(t,x) &= 0,
    && t \in (0, T), x \in (\xmin, \xmax),
  \\
  u(0,x) &= u^0(x),
    && x \in [\xmin, \xmax],
\end{aligned}
\end{equation}
with periodic boundary conditions. The functionals
\begin{subequations}
\label{eq:linear-invariants}
\begin{align}
  J_1^{\mathrm{lin}}(u) &= \int_{\xmin}^{\xmax} u,
  \\
  \label{eq:linear-invariants-energy}
  J_2^{\mathrm{lin}}(u) &= \int_{\xmin}^{\xmax} (u^2 + (\partial_x u)^2)
                         = \int_{\xmin}^{\xmax} u (\I - \partial_x^2) u,
\end{align}
\end{subequations}
are invariants of solutions of \eqref{eq:linear}, which are conserved by
semidiscretizations of the form
\begin{equation}
\label{eq:linear-SBP}
  \partial_t \vec{u} + (\I - \D2)^{-1} \D1 \vec{u} = \vec{0}.
\end{equation}
\begin{theorem}
  If $\D1$ is a periodic first-derivative SBP operator
  and $\D2$ is  a periodic second-derivative SBP operator,
  then the semidiscretization \eqref{eq:linear-SBP}
  conserves the invariants \eqref{eq:linear-invariants} of \eqref{eq:linear}.
\end{theorem}
\begin{proof}
  Using $\vec{1}^T \M (\I - \D2)^{-1} = \vec{1}^T \M$
  \cite[Lemma~2.28]{ranocha2021broad} and $\vec{1}^T \M \D1 = \vec{0}^T$
  \cite[Lemma~2.27]{ranocha2021broad} results in
  \begin{equation}
    \vec{1}^T \M \partial_t \vec{u}
    =
    - \vec{1}^T \M (\I - \D2)^{-1} \D1 \vec{u}
    =
    0.
  \end{equation}
  Using the symmetry of $\I - \D2$ with respect to the mass matrix $\M$
  and applying the SBP property \eqref{eq:D1-periodic} results in
  \begin{equation}
    \partial_t J_2^{\mathrm{lin}}(\vec{u})
    =
    2 \vec{u}^T \M (\I - \D2) \partial_t \vec{u}
    =
    - 2 \vec{u}^T \M \D1 \vec{u}
    =
    - \vec{u}^T \M \D1 \vec{u} + \vec{u}^T \D1^T \M \vec{u}
    =
    0.
    \qedhere
  \end{equation}
\end{proof}

Numerical result for the analytical solution
\begin{equation}
  u(t, x) = \sin(\pi (x - c t)), \quad c = \frac{1}{1 + 4 \pi^2},
\end{equation}
in the domain $[\xmin, \xmax] = [-1, 1]$ are shown in Figure~\ref{fig:linear}.
The spatial semidiscretization using a Fourier pseudospectral method with
$N = 2^6$ nodes is integrated in time with the fifth-order accurate Runge-Kutta
method of \cite{tsitouras2011runge} and a tolerance of $10^{-5}$.
While the conservative method has a significantly reduced amplitude error,
the phase error of the conservative and non-conservative methods are visually
indistinguishable. Consequently, the error growth rate in time for both methods
is linear. This is in accordance with analytical results available for other
linear equations such as hyperbolic systems with constant coefficients
in periodic domains; for these equations, the error can also be bounded in
time if non-periodic domains are used and the boundary conditions are imposed
appropriately
\cite{nordstrom2007error,kopriva2017error,oeffner2018error,offner2019error}.

\begin{figure}[htb]
\centering
  \begin{subfigure}{0.6\textwidth}
    \includegraphics[width=\textwidth]{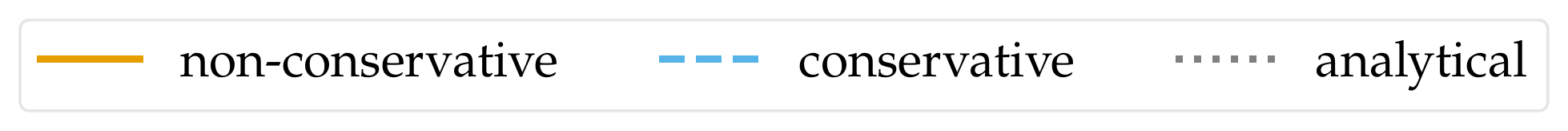}
  \end{subfigure}%
  \\
  \begin{subfigure}{0.49\textwidth}
    \includegraphics[width=\textwidth]{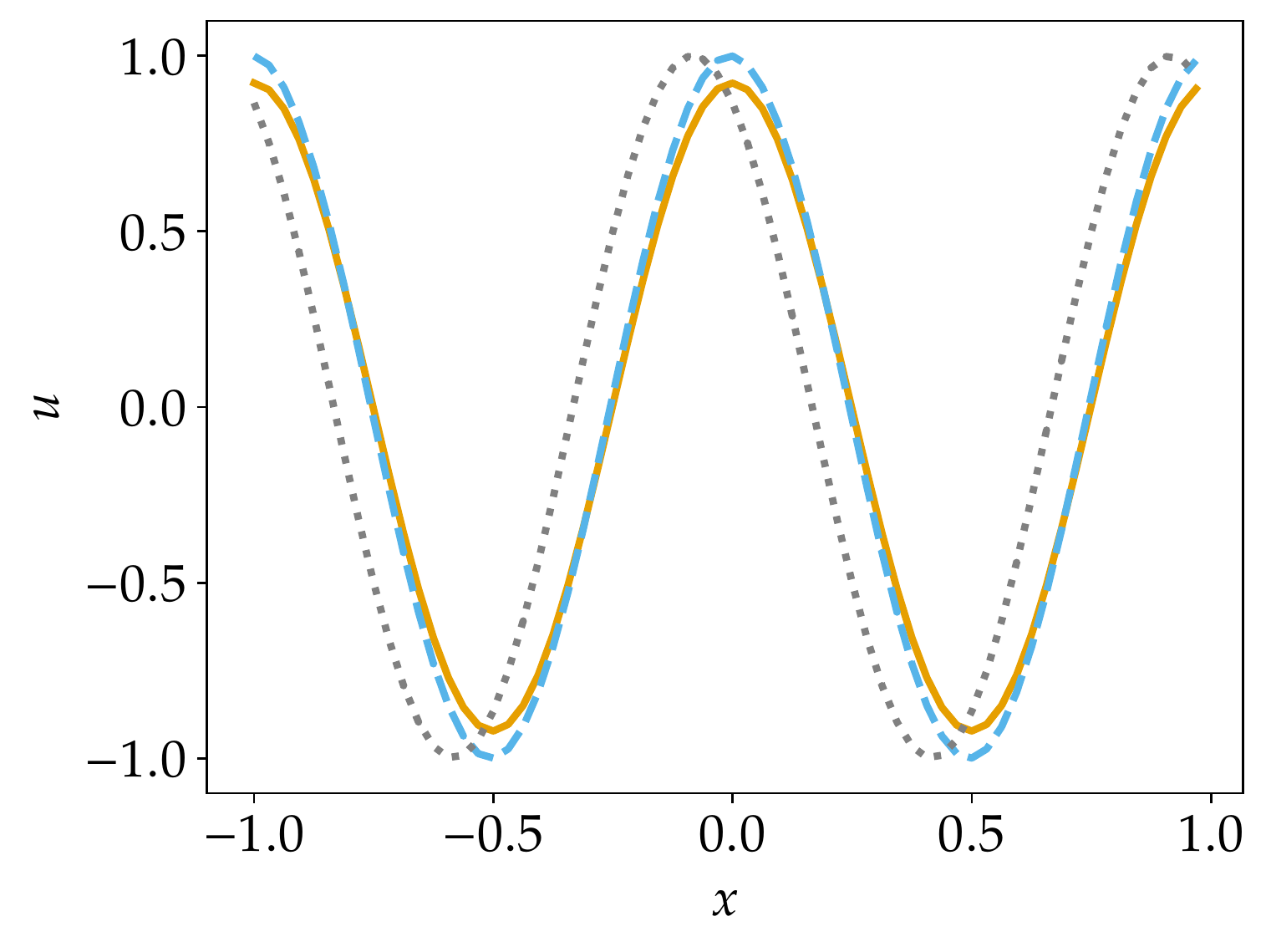}
    \caption{Numerical solution of $u$ at the final time.}
  \end{subfigure}%
  \hspace*{\fill}
  \begin{subfigure}{0.49\textwidth}
    \includegraphics[width=\textwidth]{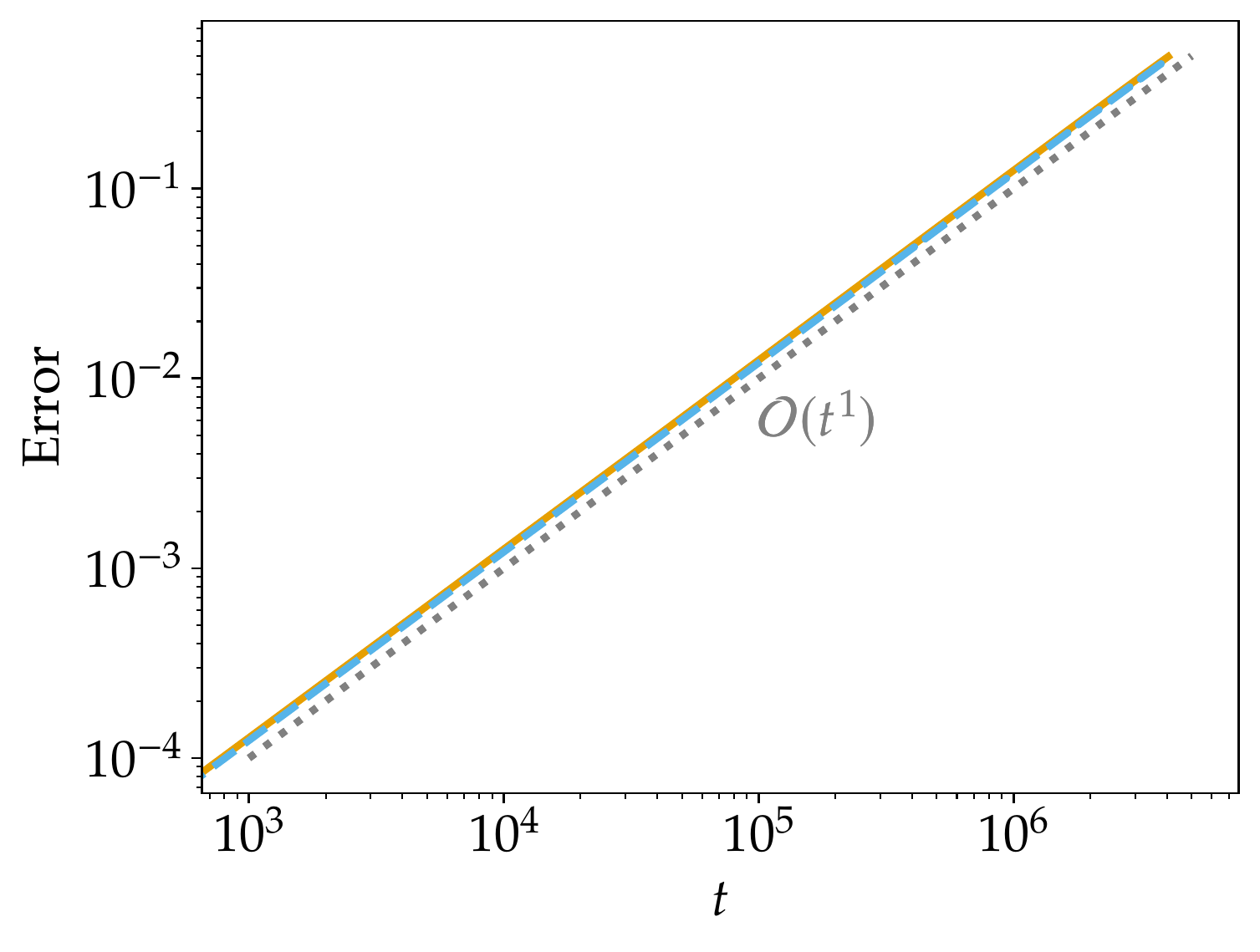}
    \caption{Error growth in time.}
  \end{subfigure}%
  \caption{Numerical results for the linear dispersive equation \eqref{eq:linear}
           using relaxation methods conserving the energy
           \eqref{eq:linear-invariants-energy}.}
  \label{fig:linear}
\end{figure}

\section{The variable-coefficient \texorpdfstring{$p$}{p}-system}
\label{sec:p-system}

The examples in Section~\ref{sec:dispersive-wave-eqs} all involve dispersive nonlinear wave equations with
constant coefficients.  The results there are in line with the results available in
the literature, all of which involve similar dispersive nonlinear wave equations.
We now turn to a very different kind of example that demonstrates the
generality of the behavior in question.  Specifically, we study a system
of nonlinear first-order hyperbolic conservation laws with spatially-varying
coefficients.  This system has no dispersive terms, and in a strict
sense has no traveling wave solutions.  It has been observed
to possess solutions (known as ``stegotons") that are similar to solitary waves,
although they change shape periodically in time; see \cite{leveque2003solitary}.

Specifically, we investigate the behavior of numerical solutions to the first-order
system
\begin{equation}
\label{eq:p-system-periodic}
\begin{aligned}
    \partial_t \epsilon(t,x) - \partial_x u(t,x) & = 0 \\
    \rho(x) \partial_t u(t,x) - \partial_x \sigma(\epsilon(t,x), x) & = 0,
\end{aligned}
\end{equation}
where $u$ is the velocity, $\rho$ the prescribed density, $\epsilon$ the strain,
and
\begin{equation}
    \sigma(\epsilon,x) = \exp(K(x)\epsilon) - 1
\end{equation}
the stress \cite{leveque2003solitary}.
Here we have used notation corresponding to elasticity, but the same
system arises in Lagrangian gas dynamics and is referred to as the $p$-system.
The energy
\begin{equation}  \label{p-system-energy}
  \int E(u(t,x), \epsilon(t,x), x) \dif x,
  \quad
  E(u, \epsilon, x) = \frac{1}{2}\rho(x) u^2 + \int_0^\epsilon \sigma(s,x) \dif s
\end{equation}
is conserved for strong solutions of \eqref{eq:p-system-periodic}.
Generically, solutions of this system give rise to discontinuities (shocks) and
a theory of weak solutions must be invoked.  If $\rho(x)$ and $\sigma(\epsilon,x)$
do not depend explicitly on $x$ then the system has no traveling wave solutions.
However, with appropriate initial
data and periodic variation of the PDE coefficients, solutions appear to be
regular for all time and the energy \eqref{p-system-energy} is conserved
\cite{2012_ketchesonleveque_periodic}.

A straightforward application of an SBP operator $\D1$ in a periodic domain
results in the semidiscretization
\begin{equation}
\label{eq:p-system-periodic-SBP}
\begin{aligned}
  \partial_t \vec{\epsilon}
  &=
  \D1 \vec{u},
  \\
  \partial_t \vec{u}
  &=
  \vec{\rho}^{-1} \D1 \vec{\sigma},
\end{aligned}
\end{equation}
where $\vec{\sigma}$ and the division by $\vec{\rho}$ are evaluated pointwise.
\begin{theorem}
\label{thm:p-system-periodic-SBP}
  Let $\D1$ be a periodic SBP first-derivative operator with diagonal mass matrix $\M$.
  Then the semidiscretization \eqref{eq:p-system-periodic-SBP} conserves
  the total mass of $\epsilon$,
  the total mass of $\rho u$,
  and the total energy $\eta = \int E$.
\end{theorem}
\begin{proof}
  For a periodic SBP operator, the rate of change of the total masses is
  \begin{equation}
  \begin{aligned}
    \partial_t \vec{1}^T \M \vec{\epsilon}
    &=
    \vec{1}^T \M \D1 \vec{u}
    =
    \vec{0},
    \\
    \partial_t \vec{1}^T \M \vec{\rho} \vec{u}
    &=
    \vec{1}^T \M \D1 \vec{\sigma}
    =
    \vec{0}.
  \end{aligned}
  \end{equation}
  Furthermore,
  \begin{equation}
    \partial_t \eta
    =
    \partial_t \vec{1}^T \M \vec{\eta}
    =
    \frac{1}{2} \partial_t \vec{1}^T \M \vec{\rho} \vec{u}^2
    + \partial_t \vec{1}^T \M \vec{\Sigma},
  \end{equation}
  where $\Sigma = \int_0^\epsilon \sigma(s, x) \dif s$. Hence,
  \begin{equation}
  \begin{aligned}
    \partial_t \eta
    &=
    \vec{1}^T \M \vec{\rho} \vec{u} \partial_t \vec{u}
    + \vec{1}^T \M \sigma \partial_t \vec{\epsilon}
    =
    \vec{1}^T \M \vec{u} \D1 \vec{\sigma}
    + \vec{1}^T \M \vec{\sigma} \D1 \vec{u}
    \\
    &=
    \vec{u}^T \M \D1 \vec{\sigma}
    + \vec{\sigma}^T \M \D1 \vec{u}
    =
    0,
  \end{aligned}
  \end{equation}
  where we have used that the mass matrix is diagonal.
\end{proof}

We use smooth coefficients given by
\begin{align*}
\rho(x) &= \frac{5}{2} - \frac{3}{2}\sin(2\pi x), \qquad K(x) = \frac{5}{2} - \frac{3}{2}\sin(2\pi x),
\end{align*}
in the domain $[0,20]$. To study the error growth in time,
we construct a stegoton solution numerically using Clawpack
\cite{clawpack561,mandli2016clawpack,ketcheson2013high,ketcheson2012pyclaw} as follows.
We start with a zero initial condition and a left boundary condition given by
\begin{align*}
  \epsilon(0,t) &= 0, \\
  u(0,t) &=
  \begin{cases}
    -0.1\left[1+\cos\left(t_0\pi\right)\right], & \mbox{if } |t_0|\leq 1, \\
       0,                                    & \mbox{otherwise},
  \end{cases}
\end{align*}
where $t_0=\frac{t-2.5}{2.5}$.  This corresponds to a moving wall
that generates a pulse that eventually becomes a train of
stegotons. We solve this problem to a late time (so that the stegotons
are well separated) on a highly-refined grid, and then isolate the
first resulting stegoton.  For more details see \cite{ranocha2021rateRepro}.

Next, we use the numerically constructed single-stegoton solution as initial
condition.  We apply an energy-conservative Fourier pseudospectral semidiscretization
\eqref{eq:p-system-periodic-SBP} with $2^8$ nodes.
We solve the resulting ODE system using the fifth-order Runge-Kutta method of
\cite{tsitouras2011runge} with adaptive time stepping and a tolerance of
$10^{-6}$. The error growth is shown in Figure~\ref{fig:p-system-Fourier-error}
for solutions obtained with relaxation (conservative) and without relaxation (non-conservative).
After an initial transient period, the error of the non-conservative scheme grows
quadratically while the corresponding conservative method results in approximately
linear error growth. The error of the non-conservative method starts to saturate
at $t \approx 10^4$ since the numerical and reference waves do not overlap
anymore, as can be seen in Figure~\ref{fig:p-system-Fourier-solution}. The
numerical solutions obtained using the non-conservative method have a clearly
visible phase error. In addition, the shape of the numerical solutions is deformed,
in particular for $u$. In contrast, the phase error of the conservative method
is negligible. Nevertheless, the shape of the numerical solutions has changed
a bit over these very long-time simulations, and small oscillations have
appeared in the tails of the solitary wave.

\begin{figure}[htb]
\centering
  \includegraphics[width=0.9\textwidth]{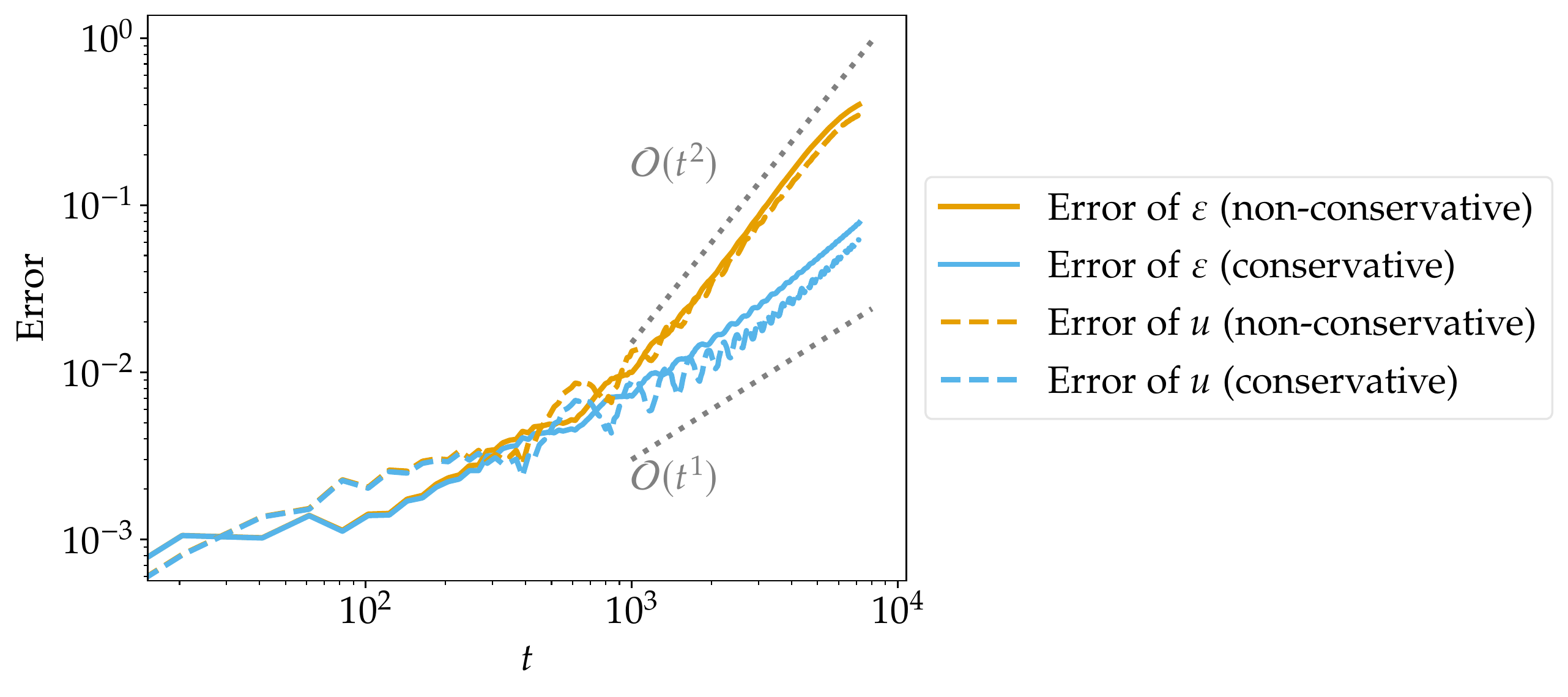}
  \caption{Error growth in time for a numerically generated stegoton solution
           of the variable-coefficient $p$-system \eqref{eq:p-system-periodic},
           using \eqref{eq:p-system-periodic-SBP} with Fourier collocation in
           space and fifth-order Runge-Kutta in time.  The conservative method
           uses relaxation to enforce conservation of \eqref{p-system-energy}.}
  \label{fig:p-system-Fourier-error}
\end{figure}

\begin{figure}[htb]
\centering
  \begin{subfigure}{0.6\textwidth}
    \includegraphics[width=\textwidth]{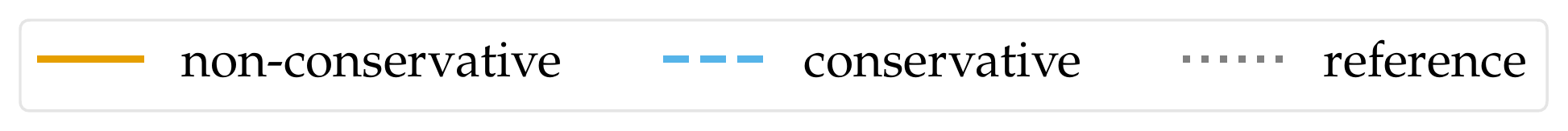}
  \end{subfigure}%
  \\
  \begin{subfigure}{0.49\textwidth}
    \includegraphics[width=\textwidth]{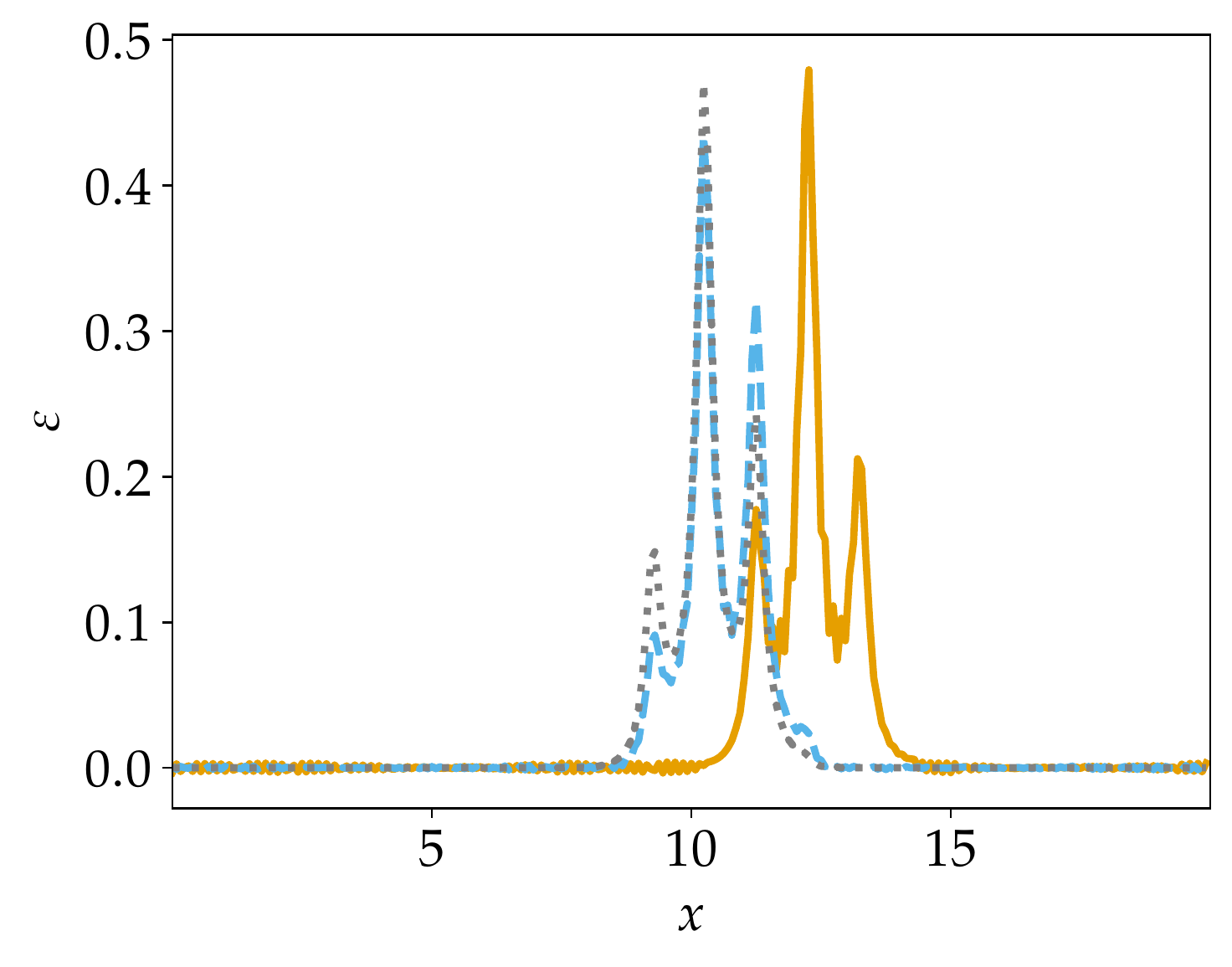}
    \caption{Numerical solution of $\epsilon$.}
  \end{subfigure}%
  \hspace*{\fill}
  \begin{subfigure}{0.49\textwidth}
    \includegraphics[width=\textwidth]{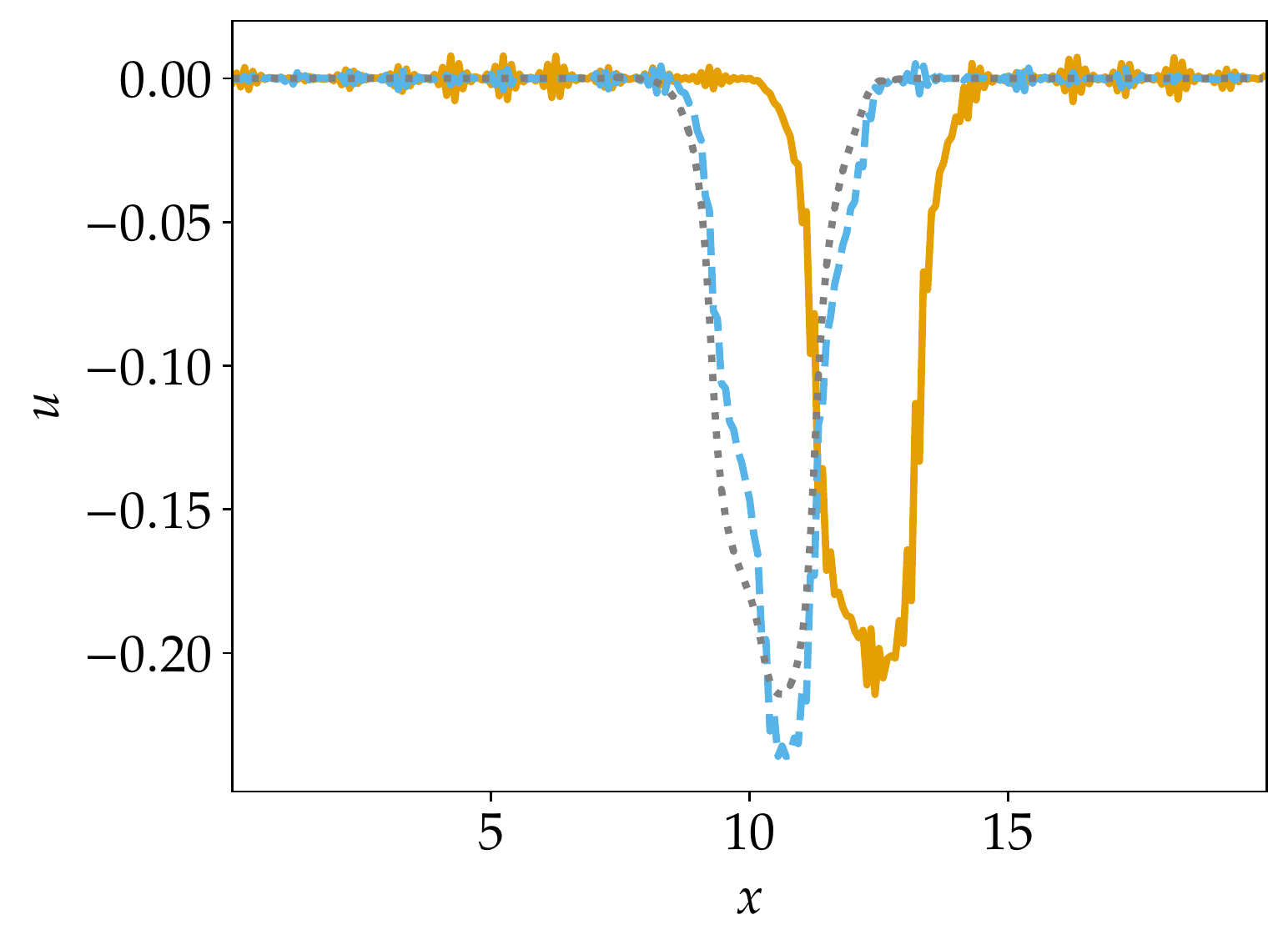}
    \caption{Numerical solution of $u$.}
  \end{subfigure}%
  \caption{Numerical solutions at the final time for a numerically
           generated stegoton solution of the variable-coefficient
           $p$-system \eqref{eq:p-system-periodic}.  The conservative
           solution is visually indistinguishable from the reference
           solution.}
  \label{fig:p-system-Fourier-solution}
\end{figure}

Although existing theory such as that reviewed in Section~\ref{sec:rel-eq} cannot
be applied to this system, it is perhaps not unreasonable to expect the behavior we have
observed.  Although stegotons are not translation invariant, they appear to be solutions
of the form
$$
    u(x,t) = \tilde{u}(x-ct, ct)
$$
where $\tilde{u}$ is periodic in its second argument; i.e.
$$
    \tilde{u}(y, z+\Omega) = \tilde{u}(y,z)
$$
where $\Omega=1$ is the period of the material coefficients $\rho(x), K(x)$.
This ``periodic translation invariance" may fulfill the same role as simple
translation invariance.  Note also that the system considered
here behaves, with a certain extreme choice of coefficients, like the discrete
Toda lattice \cite{leveque2003solitary}; for the latter system, it is known that symplectic numerical
integrators exhibit linear error growth \cite[pp. 413-414]{hairer2013geometric}.

\section{The shallow water equations}
\label{sec:swe}

We now turn to the shallow water equations in two space dimensions.  Like the
last example, this is a first-order hyperbolic system whose solutions generically
develop shock discontinuities.  As with the last example, the constant-coefficient
equations have no traveling wave solutions; however, for appropriate initial data
and varying bathymetry, numerical experiments suggest that solitary wave
solutions exist \cite{quezada2019solitary}.  Nevertheless, this example
differs in important ways from all the preceding ones and from all the results
available in the literature on error growth for dispersive nonlinear waves.  It
is a two-dimensional system, and its solitary wave solutions have a two-dimensional
structure (localized in $x$ and periodic in $y$).  Aside from the theoretical complication this involves, it also
imposes a practical challenge.  Since the exact form of the solitary wave
solutions is not known, we will need to compute an approximate solitary wave initial
condition, as we did in Section~\ref{sec:p-system}.  However, in this case we cannot
exhaustively resolve the solution and we expect that resulting
initial condition is close to, but still different from, a solitary wave.

The shallow water equations in two space dimensions with variable bathymetry $b(x,y)$ are
\begin{equation}
\label{eq:swe-2D}
\begin{aligned}
  \partial_t h
    + \partial_x (h v_x)
    + \partial_y (h v_y)
    &= 0,
  \\
  \partial_t (h v_x)
    + \partial_x (h v_x^2 + \frac{1}{2} g h^2)
    + \partial_y (h v_x v_y)
    &= - g h \partial_x b,
  \\
  \partial_t (h v_y)
    + \partial_x (h v_x v_y)
    + \partial_y (h v_y^2 + \frac{1}{2} g h^2)
    &= - g h \partial_y b,
\end{aligned}
\end{equation}
where $h(x,y,t)$ is the water height and $v_x(x,y,t), v_y(x,y,t)$ are the
$x$- and $y$-components of velocity, respectively.
The conserved energy is
\begin{align}
  E = \frac{1}{2}(h v^2 + g h^2) + g h b.
\end{align}

A two-parameter family of well-balanced, energy-conservative
semidiscretizations of the shallow water equations with variable bathymetry
was presented in \cite{ranocha2017shallow,ranocha2018thesis}. The semidiscretization
\begin{subequations}
\label{eq:swe-2D-periodic-SBP}
\begin{align}
  \partial_t \vec{h}
  &=
  - \D{1,x} \vec{h} \vec{v_x}
  - \D{1,y} \vec{h} \vec{v_y},
  \\
  \partial_t \vec{h} \vec{v_x}
  &=
  - \frac{1}{2} \D{1,x} \vec{h} \vec{v_x}^2
  - \frac{1}{2} \vec{h} \vec{v_x} \D{1,x} \vec{v_x}
  - \frac{1}{2} \vec{v_x} \D{1,x} \vec{h} \vec{v_x}
  \\&\quad
  - \frac{1}{2} \D{1,y} \vec{h} \vec{v_x} \vec{v_y}
  - \frac{1}{2} \vec{h} \vec{v_x} \D{1,y} \vec{v_y}
  - \frac{1}{2} \vec{v_y} \D{1,y} \vec{h} \vec{v_x}
  - g \vec{h} \D{1,x} (\vec{h} + \vec{b}),
  \\
  \partial_t \vec{h} \vec{v_y}
  &=
  - \frac{1}{2} \D{1,x} \vec{h} \vec{v_x} \vec{v_y}
  - \frac{1}{2} \vec{h} \vec{v_y} \D{1,x} \vec{v_x}
  - \frac{1}{2} \vec{v_x} \D{1,x} \vec{h} \vec{v_y}
  \\&\quad
  - \frac{1}{2} \D{1,y} \vec{h} \vec{v_y}^2
  - \frac{1}{2} \vec{h} \vec{v_y} \D{1,y} \vec{v_y}
  - \frac{1}{2} \vec{v_y} \D{1,y} \vec{h} \vec{v_y}
  - g \vec{h} \D{1,y} (\vec{h} + \vec{b}).
\end{align}
\end{subequations}
is a member of this family and was presented earlier in
\cite{gassner2016well,wintermeyer2017entropy}, see also
\cite{fjordholm2011well} for related methods.

In the following, we use the gravitational constant $g = 9.8$ and the smooth
bottom topography
\begin{equation}
  b(x, y) = \frac{1}{4} - \frac{1}{4} \sin(2 \pi y)
\end{equation}
in the domain $[0, 20] \times [-0.5, 0.5]$ with periodic boundary conditions.
Results in \cite{quezada2019solitary} and additional computations we
have conducted indicate that a wide range of initial conditions lead to solitary
waves that propagate along the $x$ coordinate direction, transverse to the variation
in bathymetry.  These solitary waves have a non-trivial structure in both spatial dimensions.

To study the error growth in time, we construct a solitary wave solution as follows.
Based on \cite{quezada2019solitary}, we start with an initial condition given by
\begin{align*}
  h(x,y)+b(x,y)=\eta^*+A\exp\left(-\frac{x^2}{4}\right), \qquad h v_x(x,y)= h v_y(x,y)=0,
\end{align*}
where $\eta^*=0.75$ and $A=5\times 10^{-2}$.
After propagating the initial condition up to a final time of $t=340$,
we obtain a train of solitary waves, from which we isolate the largest wave.
For more details see \cite{ranocha2021rateRepro}.
Next, we use the numerically constructed wave as initial
condition for an energy-conservative Fourier pseudospectral discretization
\eqref{eq:swe-2D-periodic-SBP} with $2^{10}$ nodes in $x$ and
$2^{6}$ nodes in $y$.
We solve the resulting ODE using the fifth-order Runge-Kutta method of
\cite{tsitouras2011runge} with adaptive time stepping and a tolerance of
$10^{-6}$. The error growth is shown in Figure~\ref{fig:swe-Fourier-error}.
In contrast to the examples considered previously, the error growth rates of
both methods are sublinear. In agreement with previous results, the error of
the conservative method still grows significantly slower, resulting in an
error at the final time (after 15 periods) that is approximately an order of magnitude
smaller.

\begin{figure}[htb]
\centering
  \includegraphics[width=0.9\textwidth]{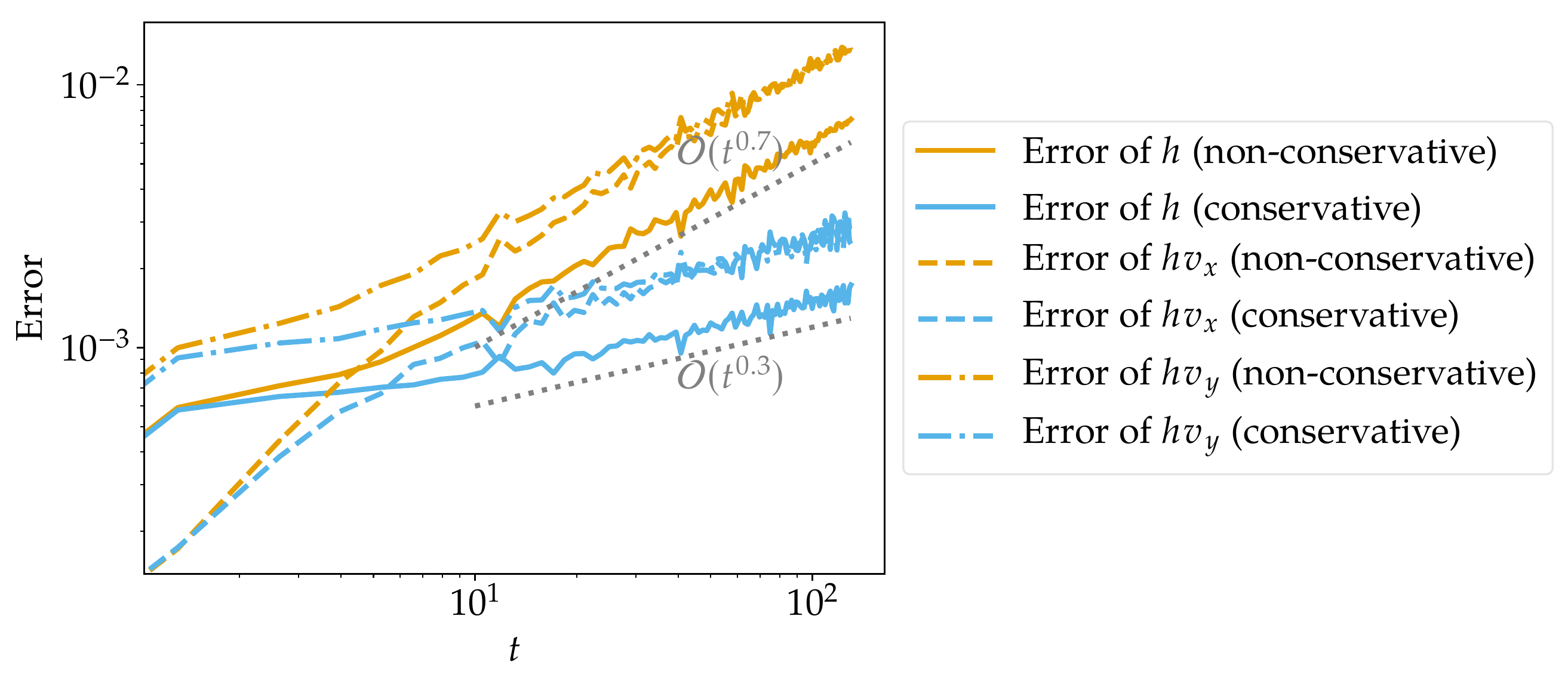}
  \caption{Error growth in time for a numerically generated solitary wave solution
           of the shallow-water equations \eqref{eq:swe-2D}.}
  \label{fig:swe-Fourier-error}
\end{figure}

Snapshots of the numerical solutions at the final time are shown in
Figures~\ref{fig:swe-Fourier-solutions} and
\ref{fig:swe-Fourier-solutions-slices}.
The conservative method advects the profiles of the perturbations well,
resulting in a numerical solution that is visually nearly indistinguishable
from the initial profile. In contrast, the non-conservative method results in
a slightly visible phase error. In addition, signs of nonlinear instabilities
in the form of high-frequency oscillations start to manifest in the y-momentum
and in reduced form also in the x-momentum and the total water height.

\begin{figure}[htb]
\centering
  \begin{subfigure}{\textwidth}
    \centering
    \includegraphics[width=\textwidth]{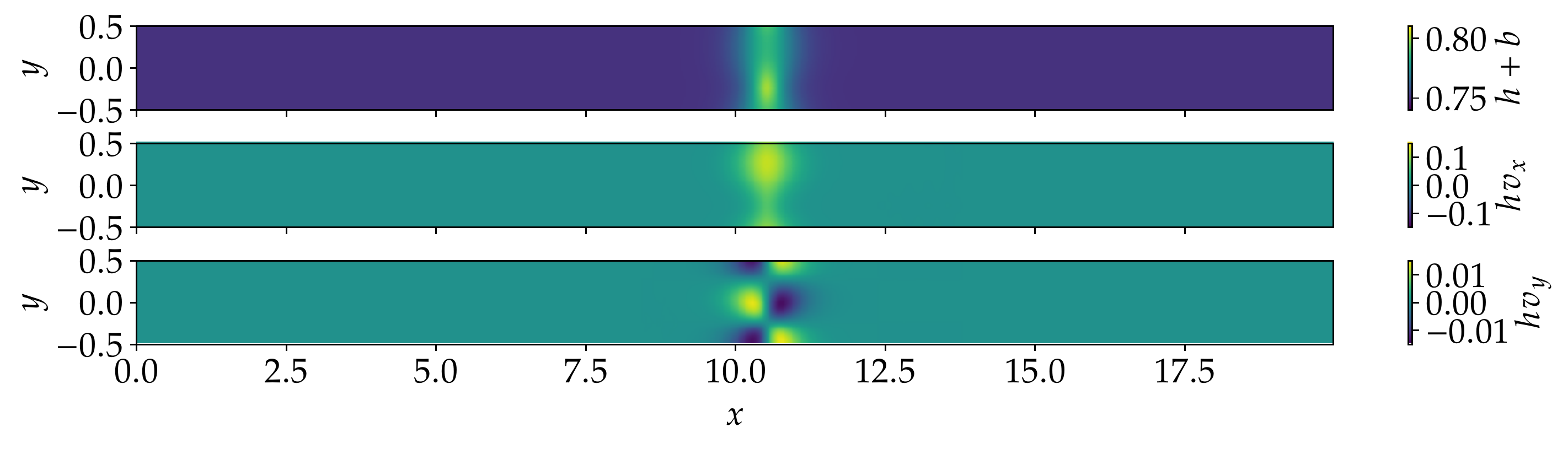}
    \caption{Initial condition.}
  \end{subfigure}%
  \\
  \begin{subfigure}{\textwidth}
    \centering
    \includegraphics[width=\textwidth]{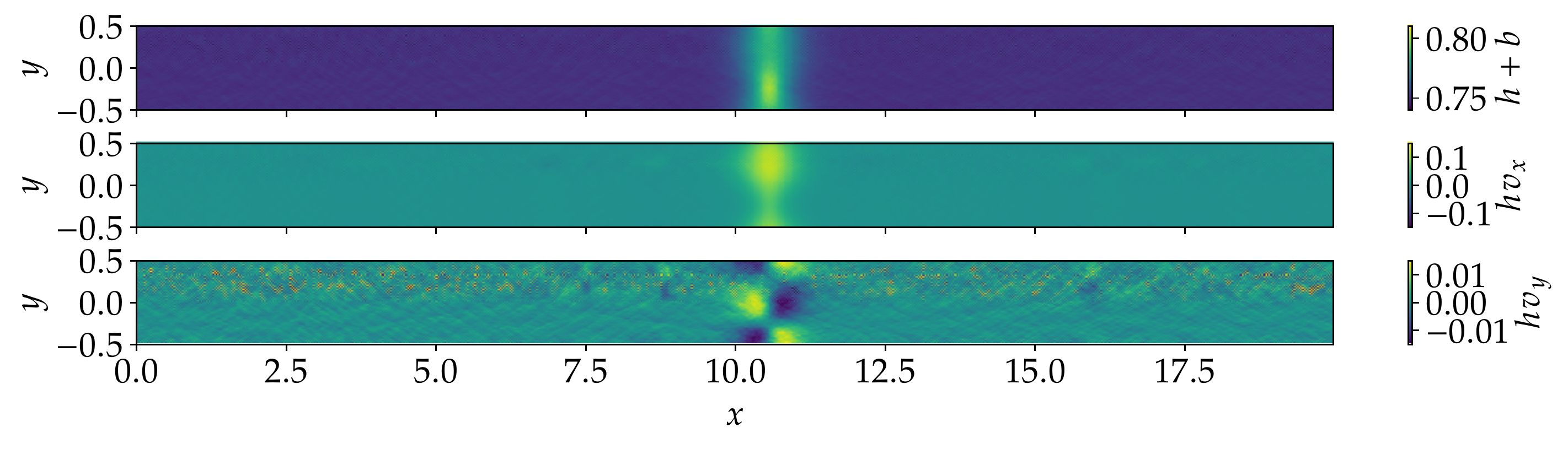}
    \caption{Non-conservative method.}
  \end{subfigure}%
  \\
  \begin{subfigure}{\textwidth}
    \centering
    \includegraphics[width=\textwidth]{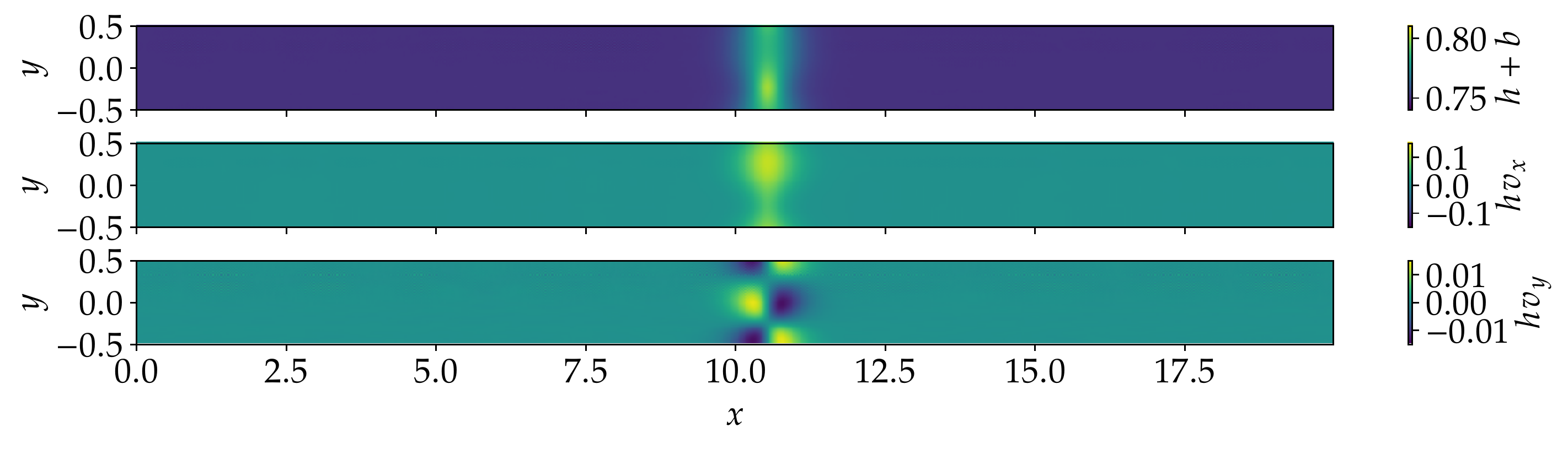}
    \caption{Conservative method.}
  \end{subfigure}%
  \caption{Snapshots of numerical solutions for a numerically generated
           solitary wave solution of the shallow-water equations \eqref{eq:swe-2D}.}
  \label{fig:swe-Fourier-solutions}
\end{figure}

\begin{figure}[htb]
\centering
  \includegraphics[width=\textwidth]{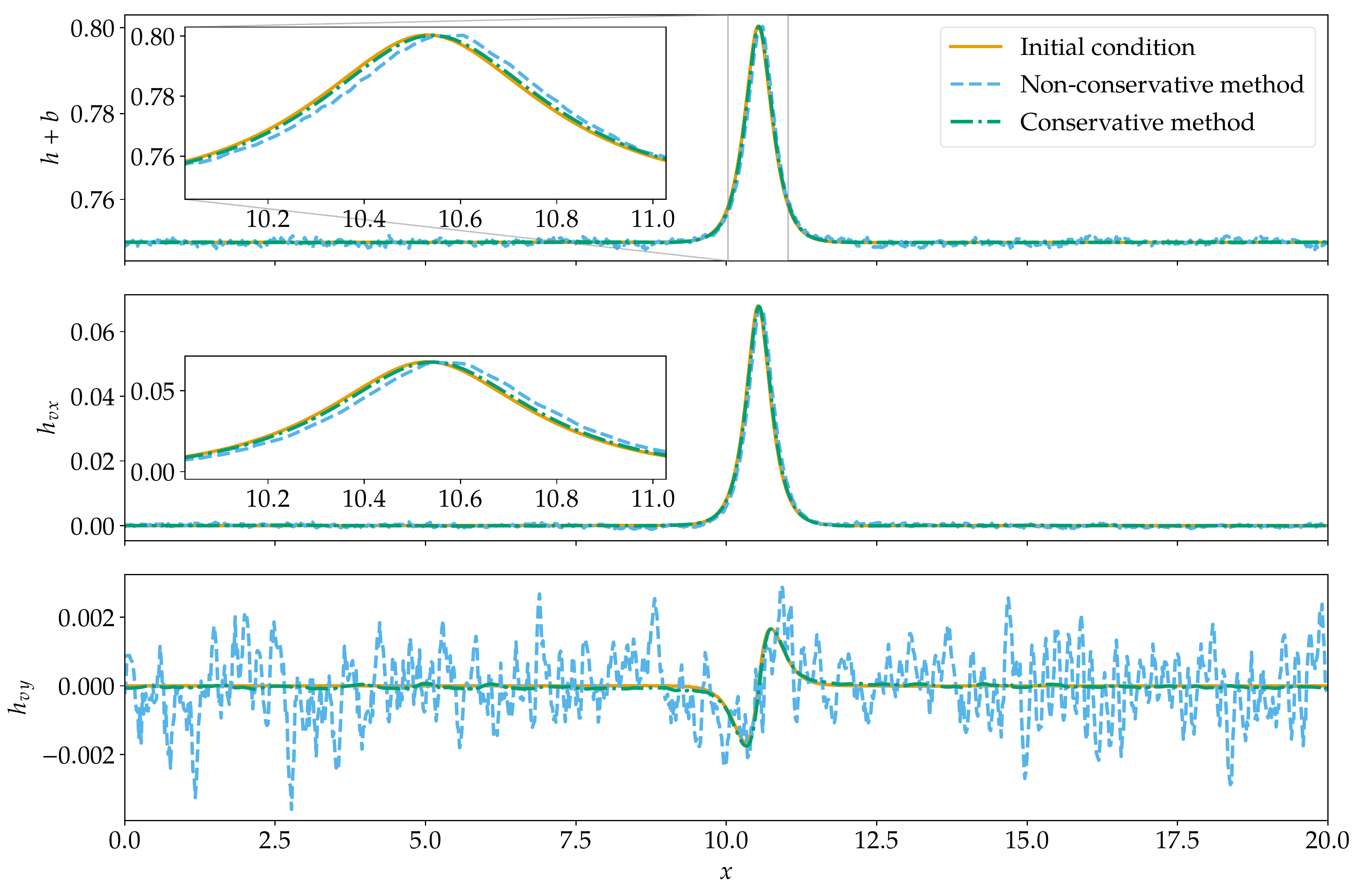}
  \caption{Slices at $y \approx -0.26$ of the numerical solutions
                shown in Figure~\ref{fig:swe-Fourier-solutions}.}
  \label{fig:swe-Fourier-solutions-slices}
\end{figure}

Possible reasons for error growth rates that are not linear/quadratic stem from
differences relative to the other equations studied so far: this system is
two-dimensional, there are no known exact solitary wave solutions, and
the computed initial solitary wave is still affected by small numerical
errors due to the cost of fully resolving the wave in two dimensions.
Additionally, the spatial resolution may not be sufficient, so spatial
discretization errors may be significant in these simulations.

Although these numerical results presented for the shallow water equations do
not match the linear/quadratic error growth rates of the other examples discussed
previously, they still show the usefulness of conservative methods. In the current implementation,
the CPU time of the conservative and non-conservative method are comparable while
the accuracy differs by an order of magnitude. To obtain similarly accurate
solutions using the non-conservative method would require an increased space/time
resolution and hence more computational resources.

\section{Summary and conclusions}
\label{sec:summary}

We have studied the error growth in time of numerical solitary wave solutions of
several systems of PDEs, focusing on the difference in behavior between conservative
and non-conservative methods.  In all cases, the general pattern is the same:
conservative methods exhibit linear error growth, while non-conservative methods
exhibit quadratic error growth.  As one might expect based on this, the
magnitude of the error itself is also much smaller for conservative methods.

We have shown that this phenomenon extends to a wide range of dispersive nonlinear wave
equations, and that this behavior seems to arise when {\em any} nonlinear invariant is
conserved, even if it is not quadratic (see Section~\ref{sec:bbm_bbm}).
Furthermore, we have shown that this behavior extends to other classes of equations,
including the $p$-system example in Section~\ref{sec:p-system}.  This
is remarkable in that the equations in question are non-dispersive and the
solutions are not traveling waves in the usual sense.  Nevertheless, a similar
advantage is obtained by using a conservative discretization.

This suggests that conservative methods may be much more efficient for solving
wave problems that possess one or more conserved functionals.  Efficiency depends
also on the cost of the method, which we have not focused on here.  In most previous
works related to this phenomenon, energy conservation was achieved through the
use of fully-implicit Runge-Kutta time integration.  With the relaxation Runge-Kutta
approach, we can instead employ essentially explicit relaxation Runge-Kutta methods
for non-stiff spatial discretizations, and diagonally implicit or linearly implicit
(Rosenbrock) methods for stiff problems.
These incur only a very small additional cost per time step (in order to solve
a scalar algebraic equation).  For the methods we have compared here, at least,
the conservative approach using relaxation is much more efficient (in terms of
computational cost for a given level of error) than the corresponding
non-conservative method.

We expect that, for the nonlinear dispersive wave equations studied, a
rigorous theoretical explanation of our results could be obtained by
perturbation analysis similar to what has been done for other similar
PDEs.  It is much less clear how to obtain a theoretical justification
for the results regarding variable-coefficient first-order hyperbolic PDEs.
A first step in this direction might be a study focused on
finite-dimensional non-autonomous Hamiltonian systems.

It is natural to ask whether the behavior studied here can be observed for
more general solutions (not consisting of a single traveling wave).
Some theoretical results are available for other classes of solutions; e.g.
for multi-soliton solutions of KdV \cite{alvarez2010multi} and quasi-periodic
solutions of KdV \cite{duran2008time}.  An extension of the current study
to more general solutions is
the subject of ongoing work, and initial tests suggest a complex variety
of behaviors.  Nevertheless, these tests indicate that conservative methods are
consistently much more accurate than their non-conservative counterparts.

\appendix
\section*{Acknowledgments}

We thank Prof. Ángel Durán for his insightful and very detailed comments on an
early draft that helped us to improve this work, and for helping us learn the
theory of relative equilibrium solutions.

Research reported in this publication was supported by the
King Abdullah University of Science and Technology (KAUST).
Funded by the Deutsche Forschungsgemeinschaft (DFG, German Research Foundation)
under Germany's Excellence Strategy EXC 2044-390685587, Mathematics Münster:
Dynamics-Geometry-Structure.

\printbibliography

\end{document}